\documentclass[a4paper,twoside]{article}
\usepackage{a4}
\usepackage{amssymb}
\usepackage{amsmath}
\usepackage{upref}
\usepackage[active]{srcltx}
\usepackage[pagebackref,colorlinks,citecolor=blue,linkcolor=blue]{hyperref}
\usepackage[dvipsnames]{color}
\allowdisplaybreaks[2] 
\newcount\minutes \newcount\hours
\hours=\time
\divide\hours 60
\minutes=\hours
\multiply\minutes -60
\advance\minutes \time
\newcommand{\klockan}{\the\hours:{\ifnum\minutes<10 0\fi}\the\minutes}
\newcommand{\tid}{\today\ \klockan}
\newcommand{\prtid}{\smash{\raise 10mm \hbox{\LaTeX ed \tid}}}
\renewcommand{\prtid}{}
\makeatletter
\def\sectionmark#1{} 
\def\subsectionmark#1{}
\newcommand{\sectnr}{\ifnum \c@secnumdepth >\z@
                 \thesection.\hskip 1em\relax \fi}
\def\@evenhead{\footnotesize\rm\thepage\hfil\leftmark\hfil\llap{\prtid}}
\def\@oddhead{\footnotesize\rm\rlap{\prtid}\hfil\rightmark\hfil\thepage}
\def\tableofcontents{\section*{Contents} 
 \@starttoc{toc}}
\makeatother
\makeatletter
\def\@biblabel#1{#1.}
\makeatother
\makeatletter
\let\Thebibliography=\thebibliography
\renewcommand{\thebibliography}[1]{\def\@mkboth##1##2{}\Thebibliography{#1}
\addcontentsline{toc}{section}{References}
\frenchspacing 
\setlength{\@topsep}{0pt}
\setlength{\itemsep}{0pt}%
\setlength{\parskip}{0pt plus 2pt}%
}
\makeatother
\makeatletter
\def\mdots@{\mathinner.\nonscript\!.%
 \ifx\next,.\else\ifx\next;.\else\ifx\next..\else
 \nonscript\!\mathinner.\fi\fi\fi}
\let\ldots\mdots@
\let\cdots\mdots@
\let\dotso\mdots@
\let\dotsb\mdots@
\let\dotsm\mdots@
\let\dotsc\mdots@
\def\vdots{\vbox{\baselineskip2.8\p@ \lineskiplimit\z@
    \kern6\p@\hbox{.}\hbox{.}\hbox{.}\kern3\p@}}
\def\ddots{\mathinner{\mkern1mu\raise8.6\p@\vbox{\kern7\p@\hbox{.}}%
    \raise5.8\p@\hbox{.}\raise3\p@\hbox{.}\mkern1mu}}
\makeatother
\makeatletter
\let\Enumerate=\enumerate
\renewcommand{\enumerate}{\Enumerate%
\setlength{\@topsep}{0pt}
\setlength{\itemsep}{0pt}%
\setlength{\parskip}{0pt plus 1pt}%
\renewcommand{\theenumi}{\textup{(\alph{enumi})}}%
\renewcommand{\labelenumi}{\theenumi}%
}
\let\endEnumerate=\endenumerate
\renewcommand{\endenumerate}{\endEnumerate\unskip}
\makeatother
\makeatletter
\def\@seccntformat#1{\csname the#1\endcsname.\quad}
\makeatother
\newcommand{\authortitle}[2]{\author{#1}\title{#2}\markboth{#1}{#2}}
\newcommand{\auth}[2]{{#1, #2.}}
\def\idxauth{\auth}
\newcommand{\art}[6]{{\sc #1, \rm #2, \it #3\/ \bf #4 \rm (#5), \mbox{#6}.}}
\newcommand{\artnopages}[5]{{\sc #1, \rm #2, \it #3\/ \bf #4 \rm (#5).}}

\newcommand{\artprep}[3]{{\sc #1, \rm #2, \rm #3.}}
\newcommand{\artin}[3]{{\sc #1, \rm #2,  in #3.}}
\newcommand{\arttoappear}[3]{{\sc #1, \rm #2, to appear in \it #3}}
\newcommand{\book}[3]{{\sc #1, \it #2, \rm #3.}}
\newcommand{\AND}{{\rm and }}
\RequirePackage{amsthm}
\newtheoremstyle{descriptive}%
  {\topsep}   
  {\topsep}   
  {\rmfamily} 
  {}          
  {\bfseries} 
  {.}         
  { }         
  {}          
\newtheoremstyle{propositional}%
  {\topsep}   
  {\topsep}   
  {\itshape}  
  {}          
  {\bfseries} 
  {.}         
  { }         
  {}          
\theoremstyle{propositional}
\newtheorem{thm}{Theorem}[section]
\newtheorem{prop}[thm]{Proposition}
\newtheorem{lem}[thm]{Lemma}
\newtheorem{cor}[thm]{Corollary}
\theoremstyle{descriptive}
\newtheorem{deff}[thm]{Definition}
\newtheorem{example}[thm]{Example}
\newtheorem{remark}[thm]{Remark}
\makeatletter
\renewenvironment{proof}[1][\proofname]{\par
  \pushQED{\qed}%
  \normalfont
  \trivlist
  \item[\hskip\labelsep
        \itshape
    #1\@addpunct{.}]\ignorespaces
}{%
  \popQED\endtrivlist\@endpefalse
}
\makeatother
\newdimen\extrawidth
\def\iintlim#1#2{\setbox0\hbox{$\scriptstyle#1$}%
        \setbox1\hbox{$\scriptstyle#2$}%
        \extrawidth=\wd1 \advance\extrawidth-\wd0
        \ifdim\extrawidth<0pt \extrawidth=0pt\fi%
        \int_{#1\kern\extrawidth \kern .5em}^{#2\kern -\wd1} \kern -.5em%
}
\def\vint{\mathop{\mathchoice%
          {\setbox0\hbox{$\displaystyle\intop$}\kern 0.22\wd0%
           \vcenter{\hrule width 0.6\wd0}\kern -0.82\wd0}%
          {\setbox0\hbox{$\textstyle\intop$}\kern 0.2\wd0%
           \vcenter{\hrule width 0.6\wd0}\kern -0.8\wd0}%
          {\setbox0\hbox{$\scriptstyle\intop$}\kern 0.2\wd0%
           \vcenter{\hrule width 0.6\wd0}\kern -0.8\wd0}%
          {\setbox0\hbox{$\scriptscriptstyle\intop$}\kern 0.2\wd0%
           \vcenter{\hrule width 0.6\wd0}\kern -0.8\wd0}}%
          \mathopen{}\int}
\newcommand{\setm}{\setminus}
\renewcommand{\emptyset}{\varnothing}
\def\cprime{{\mathsurround0pt$'$}}
\newcommand{\Cp}{{C_p}}
\newcommand{\Cone}{{C_1}}
\DeclareMathOperator{\diam}{diam}
\DeclareMathOperator{\Div}{div}
\DeclareMathOperator{\dvg}{div}
\DeclareMathOperator{\dist}{dist}
\DeclareMathOperator{\Lip}{Lip}
\DeclareMathOperator*{\osc}{osc}
\newcommand{\grad}{\nabla}
\newcommand{\bdry}{\partial}
\newcommand{\bdy}{\bdry}
\newcommand{\loc}{_{\rm loc}}
{\catcode`p =12 \catcode`t =12 \gdef\eeaa#1pt{#1}}      
\def\accentadjtext#1{\setbox0\hbox{$#1$}\kern   
                \expandafter\eeaa\the\fontdimen1\textfont1 \ht0 }
\def\accentadjscript#1{\setbox0\hbox{$#1$}\kern 
                \expandafter\eeaa\the\fontdimen1\scriptfont1 \ht0 }
\def\accentadjscriptscript#1{\setbox0\hbox{$#1$}\kern   
                \expandafter\eeaa\the\fontdimen1\scriptscriptfont1 \ht0 }
\def\accentadjtextback#1{\setbox0\hbox{$#1$}\kern       
                -\expandafter\eeaa\the\fontdimen1\textfont1 \ht0 }
\def\accentadjscriptback#1{\setbox0\hbox{$#1$}\kern     
                -\expandafter\eeaa\the\fontdimen1\scriptfont1 \ht0 }
\def\accentadjscriptscriptback#1{\setbox0\hbox{$#1$}\kern 
                -\expandafter\eeaa\the\fontdimen1\scriptscriptfont1 \ht0 }
\def\itoverline#1{{\mathsurround0pt\mathchoice
        {\rlap{$\accentadjtext{\displaystyle #1}
                \accentadjtext{\vrule height1.593pt}
                \overline{\phantom{\displaystyle #1}
                \accentadjtextback{\displaystyle #1}}$}{#1}}
        {\rlap{$\accentadjtext{\textstyle #1}
                \accentadjtext{\vrule height1.593pt}
                \overline{\phantom{\textstyle #1}
                \accentadjtextback{\textstyle #1}}$}{#1}}
        {\rlap{$\accentadjscript{\scriptstyle #1}
                \accentadjscript{\vrule height1.593pt}
                \overline{\phantom{\scriptstyle #1}
                \accentadjscriptback{\scriptstyle #1}}$}{#1}}
        {\rlap{$\accentadjscriptscript{\scriptscriptstyle #1}
                \accentadjscriptscript{\vrule height1.593pt}
                \overline{\phantom{\scriptscriptstyle #1}
                \accentadjscriptscriptback{\scriptscriptstyle #1}}$}{#1}}}}
\newcommand{\al}{\alpha}
\newcommand{\alp}{\alpha}
\newcommand{\ga}{\gamma}
\newcommand{\de}{\delta}
\newcommand{\la}{\lambda}
\newcommand{\sig}{\sigma}
\newcommand{\Om}{\Omega}
\renewcommand{\phi}{\varphi}
\newcommand{\p}{{$p\mspace{1mu}$}}
\newcommand{\R}{\mathbf{R}}
\newcommand{\Q}{\mathbf{Q}}
\newcommand{\limplus}{{\mathchoice{\raise.17ex\hbox{$\scriptstyle +$}}
                {\raise.17ex\hbox{$\scriptstyle +$}}
                {\raise.1ex\hbox{$\scriptscriptstyle +$}}
                {\scriptscriptstyle +}}}
\newcommand{\Np}{N^{1,p}}
\newcommand{\Nploc}{N^{1,p}\loc}
\newcommand{\ut}{\tilde{u}}
\newcommand{\Bh}{\widehat{B}}
\newcommand{\yh}{\hat{y}}
\newcommand{\Ga}{\Gamma}
\newcommand{\Lploc}{L^p\loc}
\newcommand{\La}{\Lambda}
\newcommand{\Xhat}{{\widehat{X}}}
\newcommand{\Ndistloc}{\Np_{\rm loc,dist}}
\newcommand{\Ndisto}{\Np_{\rm 0,dist}}
\newcommand{\Gdist}{{\mathcal G}_{\rm dist}}
\newcommand{\CpX}{{C_p^X}}
\newcommand{\wt}{\widetilde{w}}
\newcommand{\up}{u_\limplus}
\numberwithin{equation}{section}
\newcommand{\eqv}{\mathchoice{\quad \Longleftrightarrow \quad}{\Leftrightarrow}
                {\Leftrightarrow}{\Leftrightarrow}}
\newcommand{\imp}{\mathchoice{\quad \Longrightarrow \quad}{\Rightarrow}
                {\Rightarrow}{\Rightarrow}}
\newenvironment{ack}{\medskip{\it Acknowledgement.}}{}

\begin{document}

\authortitle{Anders Bj\"orn, Jana Bj\"orn
    and Nageswari Shanmugalingam}
{The Liouville theorem for \p-harmonic functions 
and quasiminimizers
with finite energy}
\author{
Anders Bj\"orn \\
\it\small Department of Mathematics, Link\"oping University, \\
\it\small SE-581 83 Link\"oping, Sweden\/{\rm ;}
\it \small anders.bjorn@liu.se
\\
\\
Jana Bj\"orn \\
\it\small Department of Mathematics, Link\"oping University, \\
\it\small SE-581 83 Link\"oping, Sweden\/{\rm ;}
\it \small jana.bjorn@liu.se
\\
\\
Nageswari Shanmugalingam
\\
\it \small  Department of Mathematical Sciences, University of Cincinnati, \\
\it \small  P.O.\ Box 210025, Cincinnati, OH 45221-0025, U.S.A.\/{\rm ;}
\it \small  shanmun@uc.edu
}

\date{}
\maketitle

\noindent{\small
{\bf Abstract} 
We show that, under certain geometric conditions,
there are no nonconstant quasiminimizers 
with finite  $p$th power
energy in a
(not necessarily complete)
metric measure space
equipped with a globally doubling measure supporting a
global \p-Poincar\'e inequality. 
The geometric conditions are 
that either (a) the measure 
has a sufficiently strong volume growth at infinity, or
(b) 
the metric space is annularly quasiconvex (or its discrete version, 
annularly chainable)
around some point in the space. 
Moreover, on
the weighted real line $\R$, we characterize
all locally doubling measures, 
supporting a local \p-Poincar\'e inequality,  for which
there exist nonconstant quasiminimizers of finite \p-energy, and 
show that a quasiminimizer is of finite \p-energy if and only if it is
bounded.
As \p-harmonic functions are quasiminimizers they are
covered by these results.
}

\bigskip

\noindent
{\small \emph{Key words and phrases}:
annular quasiconvexity,
doubling measure,
finite \p-energy,
Liouville theorem,
metric measure space,
\p-harmonic function,
Poincar\'e inequality,
quasiharmonic function,
quasiminimizer,
weak maximum principle.
}

\medskip


\section{Introduction}

The Liouville theorem in classical complex analysis states that there is no 
bounded nonconstant holomorphic function on
the entire complex plane. 
Its analogue for harmonic functions says
that there is no 
bounded (or positive) nonconstant harmonic function 
on the entire Euclidean space $\R^n$. 
This latter Liouville theorem is a consequence of the fact that 
positive
harmonic functions on the Euclidean space satisfy a Harnack type inequality. 

Harnack inequalities hold also for solutions of many nonlinear differential
equations, such as the \p-Laplace equation
\[
\Delta_p u:= \dvg(|\grad u|^{p-2}\grad u)=0, \quad 1<p<\infty,
\]
whose (continuous) solutions are \p-harmonic functions.
It then follows that every positive \p-harmonic 
function on the entire Euclidean space  $\R^n$ must be constant.
A similar conclusion holds for global solutions of the
$\mathcal{A}$-harmonic equation $\dvg \mathcal{A}(x,\grad u)=0$
with $\mathcal{A}$ of \p-Laplacian type, whose theory in  weighted
$\R^n$ has been developed in Heinonen--Kilpel\"ainen--Martio~\cite{HeKiMa}.
Note that 2-harmonic functions
on unweighted $\R^n$ are just the classical harmonic functions.
For quasilinear  equations and systems on $\R^n$ (with $p=2$), bounded
Liouville theorems and their connection to regularity of solutions
were studied in e.g.\ \cite{HildWid}, \cite{Kawohl} and \cite{Meier}.

In the setting of certain Riemannian manifolds, Harnack inequalities 
leading 
to the Liouville theorem
for positive \p-harmonic functions
were 
considered in 
Coulhon--Holopainen--Saloff-Coste~\cite{CHSC01}.
A similar Liouville theorem
on graphs, whose 
(counting) measure is globally
doubling and supports a global \p-Poincar\'e
inequality, was obtained 
in Holo\-painen--Soardi~\cite{HoSo2}.
In the last two decades, Harnack inequalities  
for \p-harmonic 
functions  and quasiminimizers were extended to metric spaces
equipped with a globally doubling measure supporting a
global \p-Poincar\'e inequality, 
see~Kinnunen--Shanmugalingam~\cite{KiSh01}, 
Bj\"orn--Marola~\cite{BMarola} and Bj\"orn--Bj\"orn~\cite{BBbook}.
Thus we know that the
Liouville theorem holds for positive quasiminimizers also 
in such settings. 
Recently, a combinatorial analogue of harmonic functions ($p=2$)
was developed in Ntalampekos~\cite{Ntalam} 
for generalized Sierpi\'nski carpets,
where the standard Poincar\'e inequality might not hold, and 
the bounded Liouville theorem for such functions was established
therein, see~\cite[Theorem~2.74]{Ntalam}. 

Since \p-harmonic functions are closely related to (and in fact are local
minimizers of) \p-energy integrals, it is natural to ask whether there 
are nonconstant \p-harmonic functions
(or quasiminimizers) with finite \p-energy
on these metric measure spaces.
Such a finite-energy Liouville theorem
for \p-harmonic functions on
Riemannian manifolds with nonnegative Ricci curvature
was obtained in
Nakauchi~\cite{Nak08} for $p \ge 2$.
See also Holopainen~\cite{Holo00}, Holopainen--Pigola--Veronelli~\cite{HoPiVe}
and Pigola--Rigoli--Setti~\cite{PiRiSe08}
for Liouville type theorems on such manifolds under various other constraints
on the \p-harmonic functions.
The manifolds in these papers, as well as in 
  \cite{CHSC01}, are all equipped with the 
Riemannian length metric and the corresponding volume measure.

Bounded, positive and $L^q$-Liouville thorems for harmonic
  functions ($p=2$) and nonlinear eigenvalue problems ($p>1$) on
  certain weighted complete Riemannian manifolds were established in e.g.\ 
\cite{Li05}, \cite{Wa3Zh} and \cite{Wu}.
See also \cite{Gar}, \cite{LiSch}, \cite{Yau} and the references
  therein for bounded and $L^q$-Liouville theorems for (sub)harmonic 
functions ($p=2$) on unweighted complete manifolds.

The primary focus of this paper is to see under 
which geometric conditions on the underlying metric measure space
the finite-energy Liouville theorem holds
for \p-harmonic functions and quasiminimizers. 
When the bounded Liouville theorem holds,
answering this question boils down to finding out whether there are 
unbounded \p-harmonic functions or quasiminimizers with finite energy.

Here, and in the rest of the paper, $1<p<\infty$ is fixed.
By  a 
\emph{quasiminimizer}
we mean a function that quasiminimizes the \p-energy, i.e.\ 
there exists
$Q\ge1$ such that for all test functions $\phi$, 
\begin{equation}  \label{eq-qmin-intro}
      \int_{\phi\ne0} g^p_u \, d\mu 
           \le Q \int_{\phi\ne0} g_{u+\phi}^p \, d\mu,
\end{equation}
where $g_u$ stands for the minimal \p-weak upper gradient of $u$,
see Definition~\ref{def-qmin}.
Quasiminimizers were introduced in Giaquinta--Giusti~\cite{GG1}, \cite{GG2}
  as a unified treatment of 
variational inequalities, elliptic
partial differential equations and
  quasiregular mappings, 
  see \cite{JBqmin15} for further references.
The following is the first main theorem of this paper.

\begin{thm} \label{thm-intro}
\textup{(Finite-energy Liouville theorem)}
Let $X$ be a metric space equipped with a globally doubling measure $\mu$
supporting a global \p-Poincar\'e inequality.
Assume that one of the following 
conditions holds\/\textup{:}
\begin{enumerate}
\item \label{a1} There is a point $x_0\in X$  and 
an exponent $\alpha\ge p$ such that
\begin{equation}    \label{eq-vol-growth}
    \limsup_{r \to \infty} \frac{\mu(B(x_0,r))}{r^\al} > 0,
\end{equation}
i.e.\ $\mu$ has volume growth of exponent $\alpha$ at infinity.
\item \label{a2}
$X$ is annularly quasiconvex around some point $x_0 \in X$
\textup{(}see Definition~\ref{def:ann-qcvx}\textup{)}. 
\item \label{a3}
$X=(\R,\mu)$, where $\mu$ is a 
globally doubling measure supporting a 
global \p-Poincar\'e inequality.
\item \label{a4}
$X$ is bounded.
\end{enumerate}
Then every quasiminimizer on $X$ with finite energy
is constant\/ \textup{(}up to a set of zero \p-capacity\/\textup{)}.
\end{thm}

In fact, in case \ref{a4}, there are no (essentially) nonconstant
quasiminimizers whatsoever, see Proposition~\ref{prop-bdd-X}.
We will show by examples that if $\mu$ supports only local versions of the
doubling condition and the \p-Poincar\'e inequality, then
the bounded, positive
and finite-energy 
Liouville theorems can fail, even for weighted $\R^n$, $n \ge 1$. 
For measures on the real line $\R$
satisfying such local
assumptions 
we will also show that, surprisingly, 
the bounded and finite-energy Liouville theorems
are equivalent, but the bounded and positive ones are not,
see Theorems~\ref{thm-weighted-R-char}
and~\ref{thm-weighted-R-char-intro-2}.
We  therefore 
distinguish between these three types of
Liouville theorems. 
  
A key ingredient in our proof of Theorem~\ref{thm-intro}
is an estimate which follows from the weak Harnack inequality and controls
the oscillation of $u$ on balls in terms of its energy,
see Lemmas~\ref{lem-key} and \ref{lem-key-lower}.
Combined with the volume growth  \eqref{eq-vol-growth} or applied to chains
of balls provided by the annular quasiconvexity, it leads to parts
\ref{a1} and \ref{a2} of Theorem~\ref{thm-intro}.
A similar, but more precise, estimate for \p-harmonic functions
with respect to ends in certain 
complete Riemannian manifolds was given 
in Holopainen~\cite[Lemma~5.3]{HoDuke}.
As a byproduct of the proof of Theorem~\ref{thm-intro} 
we obtain lower bounds for the growth of the energy and oscillation 
of nonconstant quasiminimizers on large balls, see
Corollaries~\ref{cor-energy-growth-4},
\ref{cor-energy-growth} and~\ref{cor-osc-beta-growth}.
The global estimates in this paper
can also be applied more locally 
to capture the geometry of the space in different directions
towards infinity
(so-called ends).
We pursue this line of research  in our forthcoming paper
\cite{BBShypend}.

The geometric conditions~\ref{a1} and~\ref{a2} are quite natural. 
A condition similar to~\ref{a2}, assuming that the diameters of spheres 
grow sublinearly,
was 
recently used
to prove a Liouville type theorem for harmonic functions of polynomial 
growth on certain weighted complete Riemannian manifolds,
see Wu~\cite[Theorem~1.1]{JWu}.

The annular quasiconvexity from~\ref{a2} is clearly satisfied by 
weighted $\R^n$, $n\ge2$, and the case $n=1$ is covered by~\ref{a3}.
So Theorem~\ref{thm-intro} covers all weighted $\R^n$, $n\ge 1$, 
with globally \p-admissible weights, including the setting considered in
Heinonen--Kilpel\"ainen--Martio~\cite{HeKiMa}.
In complete spaces, annular quasiconvexity
follows from sufficiently strong Poincar\'e inequalities, by
Korte~\cite[Theorem~3.3]{Ko07}.
In Lemma~\ref{lem-qconv-Riikka} we show that in noncomplete spaces, such 
Poincar\'e inequalities imply 
a discrete analogue of the annular quasiconvexity, which also implies 
the conclusion of Theorem~\ref{thm-intro}.

Note that we do not require the space $X$ to be complete.
This makes our results applicable also e.g.\
in the setting of Carnot--Carath\'eodory spaces, which in general need not 
be complete but do support a global doubling 
condition and a global $1$-Poincar\'e inequality, see
Jerison~\cite[Theorem~2.1 and Remark, p.~521]{Jer86} and 
Franchi--Lu--Wheeden~\cite{FrLuWhe}. 
In this setting, \p-harmonic functions
were first studied in Capogna--Danielli--Garofalo~\cite{CaDaGa93}, 
and are solutions
to subelliptic equations on the original Euclidean spaces.
Carnot--Carath\'eodory spaces include Heisenberg groups and are themselves
special types of metric measure spaces that satisfy our global
assumptions; see Haj\l asz--Koskela~\cite[Section~11]{HaKo}
and Remark~\ref{rmk-gu}.
Many results about \p-harmonic functions on Carnot--Carath\'eodory spaces
are thus included in the corresponding theory on metric spaces, studied
in e.g.\
Shanmugalingam~\cite{Sh-harm}, Kinnunen--Shanmugalingam~\cite{KiSh01},
Bj\"orn--Marola~\cite{BMarola} and Bj\"orn--Bj\"orn~\cite{BBbook}.

Volume growth conditions at infinity, similar to \eqref{eq-vol-growth},  
have been used to classify 
so-called parabolic and hyperbolic ends in
metric
spaces and Riemannian manifolds, see e.g.\ 
\cite{CHSC01}, \cite{Grig99}, \cite{HoDuke} and~\cite{HK01}.
They also play a role in capacity estimates for large annuli 
\cite{BBLeh1} and are related to global Sobolev embedding  theorems
\cite[Theorem~5.50]{BBbook}.

In classical conformal geometry, 
Riemann surfaces have been classified according to
the nonexistence of nonconstant harmonic functions
which are  bounded, positive and/or of finite energy,
see e.g.\ \cite{BrKi}, \cite{Pfl} and \cite{SNWCh}.
Similar results for Riemannian manifolds can be found for example
in \cite{Grig82}, \cite[Section~13]{Grig99} and \cite{SNWCh}.
This theory has been extended to include \p-harmonic functions 
in~\cite{Anc}, \cite{HoThesis}, and to the setting of metric measure
spaces in~\cite{HK01} and~\cite{HS02}. 
The studies undertaken in these papers for metric measure spaces 
did not take into account the energy of 
the global \p-harmonic functions 
as
we do here.

We consider quasiminimizers in the Liouville theorem, which means
that our results directly apply also to
solutions of the $\mathcal{A}$-harmonic equation
\begin{equation}   \label{eq-A-harm}
\dvg\mathcal{A}(x,\nabla u)=0,
\end{equation}
where $\mathcal{A}:\R^n\times \R^n\to \R^n$ is a vector field that satisfies 
certain ellipticity conditions associated with the index $p$ and a
globally \p-admissible weight $w$, as in 
Heinonen--Kilpel\"ainen--Martio~\cite{HeKiMa}.
Note that by \cite[Section~3.13]{HeKiMa}, such $\mathcal{A}$-harmonic
functions are quasiminimizers.
Since Riemannian manifolds of nonnegative curvature satisfy a global
doubling condition
and a global $1$-Poincar\'e inequality, \p-harmonic and
$\mathcal{A}$-harmonic (in the sense of \eqref{eq-A-harm}) functions on
such manifolds can be treated by Theorem~\ref{thm-intro} as well.

Another reason for  including quasiminimizers in our study
comes from geometric considerations similar to those described above.
The geometric programme of classifying metric measure spaces according 
to quasiconformal equivalences
seeks to identify two metric measure spaces as being equivalent if there 
is a quasiconformal homeomorphism between them. 
It follows from~\cite[Section~7]{HeKo98}, 
\cite[Theorem~9.10]{HKST01} and \cite[Theorem~4.1]{KMS} 
that given two 
uniformly locally Ahlfors \p-regular proper metric
spaces supporting uniformly local
\p-Poincar\'e inequalities,
any quasiconformal homeomorphism between them
induces a morphism between the corresponding classes
of quasiminimizers 
with finite energy. 
However, it
does not in general induce a morphism between the classes of \p-harmonic functions. 
Now if one of the two spaces supports a nonconstant
quasiminimizer
with finite energy but the other does not, then
there can be no quasiconformal equivalence between them. 
Thus the results developed in this paper give a 
useful tool in 
quasiconformal geometry and provide a
framework for potential-theoretic 
classifications of unbounded metric measure spaces.
Such a 
study is currently being carried out by the authors in 
\cite{BBShypend}.

On the unweighted real line
$\R$, both the volume growth condition~\eqref{eq-vol-growth},
with $\al\ge p>1$,
and the annular quasiconvexity 
fail.
At the same time, the only \p-harmonic functions on $\R$ are affine functions, 
and for such 
functions the global energy is clearly infinite unless the function 
is constant. 
Even in this simple setting, it is not trivial to show that there are no 
nonconstant \emph{quasiminimizers} with finite energy, 
but we do so in Section~\ref{sect-c} when proving
Theorem~\ref{thm-intro}\,\ref{a3}.
A similar question for the unweighted strip $\R\times[0,1]$ is adressed in 
Example~\ref{ex-RxI}.
Note that even on the unweighted real line, quasiminimizers have 
a rich theory, see e.g.\ Martio--Sbordone~\cite{MaSb}.

On weighted $\R$, equipped with a locally doubling measure
supporting a local \p-Poincar\'e inequality, we give the following
complete characterization of when
the bounded and finite-energy Liouville theorems hold.
Under the conditions considered in this paper, any
quasiminimizer has a continuous representative, 
which is called \emph{quasiharmonic}.
(The discussion above should correctly be for these quasiharmonic
representatives.)

\begin{thm}\label{thm-weighted-R-char}
Let $\mu$ be  a locally 
doubling measure on $\R$ supporting a local \p-Poincar\'e inequality.
Then the following are equivalent\/\textup{:}
\begin{enumerate}
\item \label{c-a}
There exists a bounded nonconstant \p-harmonic function on $(\R,\mu)$.
\item \label{c-b}
There exists a nonconstant \p-harmonic function with finite energy 
on $(\R,\mu)$.
\item \label{c-c}
There exists a  bounded nonconstant quasiharmonic function on $(\R,\mu)$.
\item \label{c-d}
There exists a nonconstant quasiharmonic function with finite energy 
on $(\R,\mu)$.
\item \label{c-e}
There is a weight $w$ such that $d\mu = w \, dx$ and
\begin{equation} \label{eq-R-bdd-L}
    \int_{-\infty}^\infty w^{1/(1-p)} \, dx < \infty.
\end{equation}
\end{enumerate}
\end{thm}
\medskip

We also show that under the local assumptions of 
Theorem~\ref{thm-weighted-R-char}, a quasiharmonic function
on weighted $\R$
is bounded if and only if it has finite energy, see 
Proposition~\ref{prop-finite-energy-bdd}.
This equivalence may be of independent interest, in addition to implying
the equivalence of the bounded and the finite-energy Liouville theorems 
on weighted~$\R$.

Examples~\ref{ex-bdd-inf-energy} resp.~\ref{ex:tree} show that
on some spaces there exist global \p-harmonic functions that are bounded but 
without finite
energy, and vice versa. 
For examples of Riemannian manifolds
where the finite-energy Liouville theorem for harmonic functions holds but 
not the bounded Liouville theorem, we refer to
Sario--Nakai--Wang--Chung~\cite[Section~1.2]{SNWCh}.

The Liouville theorem is often given for positive functions,
but this is not always equivalent to the bounded (or finite-energy)
Liouville theorem,
as demonstrated by the following result
(together with Theorem~\ref{thm-weighted-R-char}).

\begin{thm} \label{thm-weighted-R-char-intro-2}
Let $\mu$ be  a locally 
doubling measure on $\R$ supporting a local \p-Poincar\'e inequality.
Then the following are equivalent\/\textup{:}
\begin{enumerate}
\item \label{S-bdd}
There exists a positive nonconstant \p-harmonic function on $(\R,\mu)$.
\item \label{S-qmin-bdd}
There exists a positive nonconstant quasiharmonic function on $(\R,\mu)$.
\item \label{S-integral}
There is a weight $w$ such that $d\mu = w \, dx$ and
\begin{equation} \label{eq-R-pos-L}
    \min\biggl\{\int_{-\infty}^0 w^{1/(1-p)} \, dx,
     \int_0^\infty w^{1/(1-p)} \, dx
    \biggr\}< \infty.
\end{equation}
\end{enumerate}
\end{thm}
\medskip

If $d\mu = w\,dx$ on $\R^n$ and $w$ is any positive  function
which is locally bounded from above and away from zero,
then it is easy to see that 
$\mu$ is locally doubling
and supports a local $1$-Poincar\'e inequality.
It follows that, for $n=1$, one can easily construct
weights such that \eqref{eq-R-pos-L} holds but \eqref{eq-R-bdd-L} fails.

The paper is organized as follows.
In Section~\ref{sect-prelim} we provide the necessary background about 
Sobolev type spaces on metric spaces.
In Section~\ref {sect-qmin} we discuss (quasi)minimizers and \p-harmonic
functions in spaces equipped with a locally doubling measure supporting
a local \p-Poincar\'e inequality.
Since it
is not assumed that the underlying metric space is complete,
the choice of test functions in~\eqref{eq-qmin-intro} plays a crucial role.
We prove a general weak maximum principle,
which despite its name does not follow from the strong maximum principle.
We also show that in bounded
spaces, all global quasiharmonic functions are locally constant;
from which Theorem~\ref{thm-intro}\,\ref{a4} follows (under the
global assumptions therein).

Sections~\ref{sect-a}--\ref{sect-c} are devoted to the proofs of
\ref{a1}--\ref{a3} of Theorem~\ref{thm-intro}, respectively.
Moreover, in Section~\ref{sect-pf-main-a2} we discuss connectivity properties
of the space, including 
a discrete version
of 
annular quasiconvexity.
Growth estimates for the energy and the oscillation of nonconstant 
quasiharmonic functions are also proved therein.
Section~\ref{sect-c} contains a rather exhaustive study of quasiharmonic
functions and functions with finite energy on $\R$, equipped with a
locally doubling measure supporting a local \p-Poincar\'e inequality,
leading up to the proofs of 
Theorems~\ref{thm-weighted-R-char}
and~\ref{thm-weighted-R-char-intro-2}.
Theorem~\ref{thm-intro}\,\ref{a3} is then a direct consequence of these
considerations.

We conclude the paper in Section~\ref{sect-further-ex} by showing that
the finite-energy Liouville theorem holds in the unweighted infinite strip 
$\R\times[0,1]$ and that it fails in a weighted binary tree, see
Examples~\ref{ex-RxI} and~\ref{ex:tree}.
The latter example also produces an unbounded \p-harmonic function
with 
finite energy.

\begin{ack}
The first two authors were supported by the Swedish Research Council
grants 2016-03424 and 621-2014-3974, respectively. 
The third
author was partially supported by the NSF grants DMS-1200915 and DMS-1500440. 
Part of the
research was done during several visits of the third author to 
Link\"oping University in 2015--18.
During parts of 2017 and 2018 she was also a guest 
professor at Link\"oping University,
partially
funded by the Knut and Alice Wallenberg Foundation;
she
thanks them for their kind support and hospitality.
\end{ack}

\section{Preliminaries}
\label{sect-prelim}

We assume throughout the paper that $1 <  p<\infty$ 
and that $X=(X,d,\mu)$ is a metric space equipped
with a metric $d$ and a positive complete  Borel  measure $\mu$ 
such that $0<\mu(B)<\infty$ for all balls $B \subset X$.
For proofs of the facts stated in this section 
we refer the reader to Bj\"orn--Bj\"orn~\cite{BBbook} and
Heinonen--Koskela--Shanmugalingam--Tyson~\cite{HKST}.

A \emph{curve} is a continuous mapping from an interval.
We will only consider curves which are nonconstant, compact and
rectifiable,
i.e.\ of finite length.
A curve can thus be parameterized by its arc length $ds$. 
A property holds for \emph{\p-almost every curve}
if the curve family $\Ga$ for which it fails has zero \p-modulus,
i.e.\ there is a Borel function $0 \le\rho\in L^p(X)$ such that
$\int_\ga \rho\,ds=\infty$ for every $\ga\in\Ga$.

\begin{deff} \label{deff-ug}
A measurable function $g:X \to [0,\infty]$ is a \emph{\p-weak upper gradient}
of $u:X \to [-\infty,\infty]$ if for \p-almost every curve
$\gamma: [0,l_{\gamma}] \to X$,
\[ 
        |u(\gamma(0)) - u(\gamma(l_{\gamma}))| \le \int_{\gamma} g\,ds,
\]
where the left-hand side is considered to be $\infty$
if at least one of the terms therein is $\pm \infty$.
\end{deff}

The \p-weak upper gradients were introduced in
Koskela--MacManus~\cite{KoMc}, 
see also Heinonen--Koskela~\cite{HeKo98}. 
If $u$ has a \p-weak upper gradient in $\Lploc(X)$, then
it has a \emph{minimal \p-weak upper gradient} $g_u \in \Lploc(X)$
in the sense that for every \p-weak upper gradient $g \in \Lploc(X)$ 
of $u$ we have
$g_u \le g$ a.e., see Shan\-mu\-ga\-lin\-gam~\cite{Sh-harm}. 
The minimal \p-weak upper gradient is well defined
up to a set of measure zero.
Note also that 
$g_u=g_v$ a.e.\ in $\{x \in X : u(x)=v(x)\}$,
in particular $g_{\min\{u,c\}}=g_u \chi_{\{u < c\}}$ a.e., for $c \in \R$.

Following Shanmugalingam~\cite{Sh-rev}, 
we define a version of Sobolev spaces on $X$. 

\begin{deff} \label{deff-Np}
For a measurable function $u:X\to [-\infty,\infty]$, let 
\[
        \|u\|_{\Np(X)} = \biggl( \int_X |u|^p \, d\mu 
                + \inf_g  \int_X g^p \, d\mu \biggr)^{1/p},
\]
where the infimum is taken over all \p-weak upper gradients $g$ of $u$.
The \emph{Newtonian space} on $X$ is 
\[
        \Np (X) = \{u: \|u\|_{\Np(X)} <\infty \}.
\]
\end{deff}

In this paper we assume that functions in $\Np(X)$
are defined everywhere (with values in $[-\infty,\infty]$),
not just up to an equivalence class in the corresponding function space.
The space $\Np(X)/{\sim}$, where  $u \sim v$ if and only if $\|u-v\|_{\Np(X)}=0$,
is a Banach space and a lattice, 
see~\cite{Sh-rev}. 
For a measurable set $E\subset X$, the Newtonian space $\Np(E)$ is defined by
considering $(E,d|_E,\mu|_E)$ as a metric space in its own right.

\begin{deff}
The (Sobolev) \emph{capacity} of a set $E\subset X$  is the number 
\[
\Cp(E)=\CpX(E) =\inf_u    \|u\|_{\Np(X)}^p,
\]
where the infimum is taken over all $u\in \Np (X) $ such that $u=1$ on $E$.
\end{deff}

A property is said to hold \emph{quasieverywhere}
(q.e.)\ if the set of all points 
at which the property
fails has $\Cp$-capacity zero. 
The capacity is the correct gauge 
for distinguishing between two Newtonian functions. 
If $u \in \Np(X)$, then $u \sim v$ if and only if $u=v$ q.e.
Moreover, if $u,v \in \Nploc(X)$ and $u= v$ a.e., then $u=v$ q.e.

We let $B=B(x,r)=\{y \in X : d(x,y) < r\}$ denote the ball
with centre $x$ and radius $r>0$, and let $\la B=B(x,\la r)$.
We assume throughout the paper that balls are open.
In metric spaces it can happen that
balls with different centres and/or radii denote the same set. 
We will however adopt the convention that a ball $B$ comes with
a predetermined centre $x_B$ and radius $r_B$. 
In this generality, it can happen that  $B(x_0,r_0) \subset B(x_1,r_1)$
even when $r_0 > r_1$, and in
disconnected spaces also when $r_0 > 2 r_1$.
If $X$ is connected, then $B(x_0,r_0) \subset B(x_1,r_1)$ with
$r_0 >2r_1$ is possible only when  $B(x_0,r_0)= B(x_1,r_1)=X$.

We shall use the following local assumptions 
introduced in Bj\"orn--Bj\"orn~\cite{BBsemilocal}.

\begin{deff} \label{def-local-doubl-mu}
The measure \emph{$\mu$ is doubling within $B(x_0,r_0)$}
if there is $C>0$ (depending on $x_0$ and $r_0$)
such that 
\[ 
\mu(2B)\le C \mu(B)
\] 
for all balls $B \subset B(x_0,r_0)$.

We say that $\mu$ is \emph{locally doubling} (on $X$) if 
for every $x_0 \in X$ there is some $r_0>0$ 
(depending on $x_0$) such that $\mu$ is doubling within $B(x_0,r_0)$.

If $\mu$ is doubling within every ball $B(x_0,r_0)$, then 
it is \emph{semilocally doubling},
and if moreover $C$ is independent of $x_0$ and $r_0$,
then $\mu$ is \emph{globally doubling}.
\end{deff}

\begin{deff} \label{def-PI}
The 
\emph{\p-Poincar\'e inequality holds within $B(x_0,r_0)$} 
if there are constants $C>0$ and $\lambda \ge 1$ (depending on $x_0$ and $r_0$)
such that for all balls $B\subset B(x_0,r_0)$, 
all integrable functions $u$ on $\la B$, and all 
\p-weak upper gradients $g$ of $u$, 
\begin{equation}    \label{eq-def-local-PI}
        \vint_{B} |u-u_B|\,d\mu
        \le C r_B \biggl( \vint_{\lambda B} g^{p} \,d\mu \biggr)^{1/p},
\end{equation}
where $u_B:= \vint_B u\,d\mu = \mu(B)^{-1}\int_B u\,d\mu$.

We also say that $X$ (or $\mu$) supports 
a \emph{local \p-Poincar\'e inequality} (on $X$) if
for every $x_0 \in X$ there is some $r_0>0$ (depending on $x_0$) 
such that the \p-Poincar\'e inequality holds within $B(x_0,r_0)$.

If the \p-Poincar\'e inequality holds within every ball $B(x_0,r_0)$,
then $X$ supports a \emph{semilocal \p-Poincar\'e inequality}.
If moreover $C$ and $\la$ are independent of $x_0$ and $r_0$,
then $X$ supports a \emph{global \p-Poincar\'e inequality}.
\end{deff}

\begin{remark} \label{rmk-semilocal}
If $X$ is 
proper 
(i.e.\ every closed and bounded subset of $X$ is compact)
and connected,
and $\mu$ is locally doubling
and supports a local \p-Poincar\'e inequality,
then 
$\mu$ is semilocally doubling
and supports a semilocal \p-Poincar\'e inequality,
by Bj\"orn--Bj\"orn~\cite[Proposition~1.2 and Theorem~1.3]{BBsemilocal}.
This in particular applies to $\R^n$ equipped with the Euclidean distance
and any measure satisfying the local assumptions.
\end{remark}

\begin{remark} \label{rmk-gu}
If $X$ is locally compact and supports a global \p-Poincar\'e inequality
and $\mu$ is globally doubling, then 
$g_u=\Lip u$ a.e.\ for 
Lipschitz functions $u$ on $X$, 
by Theorem~6.1 in Cheeger~\cite{Cheeg} together with Lemma~8.2.3 in 
Hei\-no\-nen--Kos\-ke\-la--Shan\-mu\-ga\-lin\-gam--Ty\-son~\cite{HKST}
(or Theorem~4.1 in Bj\"orn--Bj\"orn~\cite{BBnoncomp}).
(Here $\Lip u$ is the upper pointwise dilation of $u$, also
called the local upper Lipschitz constant.)
Moreover, Lipschitz functions
are dense in $\Np(X)$, see Shan\-mu\-ga\-lin\-gam~\cite{Sh-rev}.

Hence if $X= \R^n$, equipped with a \p-admissible measure
as in Heinonen--Kilpel\"ainen--Martio~\cite{HeKiMa},
then $g_u= |\nabla u|$ for $u \in \Np(X)$, where
$\nabla u$ is the weak Sobolev gradient from~\cite{HeKiMa}.
The corresponding identities for the gradients hold also on
Riemannian manifolds and Carnot--Carath\'eodory spaces
equipped with their natural measures; 
see Haj\l asz--Koskela~\cite[Section~11]{HaKo}. 
\end{remark}

If $X$ is connected  
(which follows from any semilocal Poincar\'e inequality,
see e.g.\ the proof of Proposition~4.2 in \cite{BBbook}), 
then the global doubling property 
implies that there are positive constants 
$\sigma\le s$  and $C$ such that 
\begin{equation}\label{eq:mass-bound-exp}
  \frac{1}{C} \Bigl(\frac{r}{R}\Bigr)^s \le 
  \frac{\mu(B(x,r))}{\mu(B(x,R))}
  \le C \Bigl(\frac{r}{R}\Bigr)^\sigma
\end{equation}
whenever $x\in X$ and $0<r\le R < 2 \diam X$.
Example~2.3 in 
Adamowicz--Bj\"orn--Bj\"orn--Shan\-mu\-ga\-lin\-gam~\cite{ABBS}
shows that $\sigma$ may need to be 
close to $0$.

Fixing $r>0$ and letting $R\to\infty$ in~\eqref{eq:mass-bound-exp} 
shows that $X$ has volume growth 
\eqref{eq-vol-growth} of exponent $\al=\sigma$ at infinity,
but it is often possible to have a larger choice of $\alpha$. 
At the same time, necessarily $\alpha\le s$.
It is easy to see that \eqref{eq-vol-growth}
is independent of $x_0$. 
The set of all possible $\alpha$ in \eqref{eq-vol-growth} is an interval
which may be open or closed at the right endpoint.
When the right endpoint does not belong to the interval, there 
is no optimal choice of $\alpha$.

A measurable function $u$ 
is of \emph{finite energy} on an open set $\Om$ if it has a \p-weak  
upper gradient in $L^p(\Om)$,
in which case its \emph{energy} on $\Om$ is given by $\int_\Om g_u^p \, d\mu$.
It follows from 
\cite[Proposition~4.14]{BBbook} or \cite[Lemma~8.1.5 and Theorem~9.1.2]{HKST}
that if $\mu$ is locally doubling and supports a local \p-Poincar\'e 
inequality, then functions with finite energy on $\Om$ belong to 
$\Np\loc(\Om)$.
Similar arguments can be used to show that under semilocal assumptions,
functions with finite energy on $\Om$ belong to 
the space $\Ndistloc(\Om)$.
See Section~\ref{sect-qmin} below for the definitions of 
$\Np\loc(\Om)$ and $\Ndistloc(\Om)$.

\section{Quasiminimizers and their test functions}
\label{sect-qmin}

\emph{We assume in this section that $\mu$ is locally doubling
  and supports a local \p-Poincar\'e inequality.
We will take extra care to avoid the requirement that the
metric space is complete or even locally compact.}

\medskip

Let $\Om \subset X$ be open.
We say  that $f \in \Nploc(\Om)$ if
for every $x \in \Om$ there exists $r_x>0$ such that 
$B(x,r_x)\subset\Om$ and $f \in \Np(B(x,r_x))$.
Traditionally, e.g.\ in $\R^n$ and other complete spaces, 
a quasiminimizer $u$ on 
$\Om$ is
required to belong to the local space $\Nploc(\Om)$
and the quasiminimizing property is tested by sufficiently smooth (e.g.\ 
Lipschitz or Sobolev) \emph{test functions} $\phi$ 
with compact support in $\Om$
(or with zero boundary values) as follows:
\[ 
      \int_{\phi\ne0} g^p_u \, d\mu 
           \le Q \int_{\phi\ne0} g_{u+\phi}^p \, d\mu.
\] 
When $X$ is noncomplete there are several natural
choices corresponding to $\Nploc(\Om)$ as well
as several choices of natural test function spaces,
and contrary to the complete case these do not all lead to equivalent definitions.
Thus we might obtain different classes of quasiminimizers by 
considering different test classes of 
$\phi$
and by requiring that $u$ belongs to 
various choices of local Newtonian spaces. 
See 
Bj\"orn--Bj\"orn~\cite[Section~6]{BBnoncomp} 
and 
Bj\"orn--Marola~\cite{BMarola} 
for 
related discussions.

Our choice of test functions
is based on the desire to have as large as possible collection of
quasiminimizers while retaining potential-theoretic properties such as 
maximum principles and weak Harnack inequalities for such quasiminimizers. 
For instance, insisting on compact support could lead to
a very small class of test functions if $X$ is not locally compact, 
and then properties such as 
the Harnack inequality 
and maximum principles might fail,
see Examples~\ref{ex-(-1,1)} and~\ref{ex-Rn-Qn} below.

We follow the notation in \cite{BBnoncomp} and define
\begin{align*}
\Np_0(\Om) &= \{\phi|_\Om: \phi\in\Np(X) \text{ and } 
\phi=0 \text{ in }X\setm \Om\}, \\ 
\Gdist(\Om) &= 
\{G\subset\Om: G \text{ is bounded and open, and } \dist(G,X\setm\Om)>0 \}, \\
\Ndistloc(\Om)& = \{u:\Om\to [-\infty,\infty]: u\in \Np(G) \text{ for all }
 G \in \Gdist(\Om)\}, \\
\Ndisto(\Om) &= \overline{ \{\phi:\Om\to [-\infty,\infty]: \phi\in \Np_0(G) 
\text{ for some } G \in \Gdist(\Om)\}},  
\end{align*}
where the closure is with respect to the Sobolev 
norm $\Vert{\,\cdot\,}\Vert_{N^{1,p}(X)}$.
We emphasize that the class $\Gdist(\Om)$ depends not only on $\Om$
but also on the ambient metric space~$X$.
Here we adopt the convention that 
$\dist(G,\emptyset)>0$ for all $G$.
In particular, this means that the requirement $\dist(G,X\setm\Om)>0$
is trivially satisfied for all $G$ when $\Om=X$. 

\begin{deff} \label{def-qmin}
A function $u \in \Ndistloc(\Om)$ is a \emph{quasiminimizer} in $\Om$ if
there exists $Q\ge1$ such that
\begin{equation} \label{eq-deff-qmin}
     \int_{\phi \ne 0} g^p_u \,d\mu \le Q  \int_{\phi \ne 0} g^p_{u+\phi} \,d\mu
\end{equation}
for all $\phi \in \Ndisto(\Om)$. 
If $Q=1$ in~\eqref{eq-deff-qmin}, then $u$ is 
a \emph{minimizer}. 
\end{deff}

Any quasiminimizer can be modified on a set of capacity zero
so that it becomes 
continuous
(by which we mean real-valued continuous in this paper).
This follows from the results in 
Kinnunen--Shan\-mu\-ga\-lin\-gam~\cite[p.~417]{KiSh01}.
The assumptions in~\cite{KiSh01} are global and guarantee local
H\"older continuity of quasiminimizers.
See Bj\"orn--Bj\"orn~\cite[Theorem~6.2]{BBnoncomp} 
for how the arguments apply in our situation.
Such a continuous representative  is called 
\emph{quasiharmonic} or, for $Q=1$, \emph{\p-harmonic}.
The strong maximum principle,
saying that if a quasiharmonic function attains its maximum in a domain then 
it is constant therein,
also holds,
see 
\cite[Corollary~6.4]{KiSh01} and 
\cite[Theorem~6.2]{BBnoncomp}.

The following example shows 
why it is essential  
to use the nonstandard space $\Ndistloc(\Om)$. 
Namely, requiring only $u\in\Nploc(\Om)$ 
in the definition of quasiminimizers would cause problems
with Harnack inequalities and the weak maximum principle.

\begin{example} \label{ex-Nploc}
Consider the locally compact space $X=\R^n \setm \{0\}$ and let $1<p< n$.
The function $u(x)=|x|^{(p-n)/(p-1)}$ is \p-harmonic in $X$ as 
an open  subset of $\R^n$.
Since $C_p^{\R^n}(\{0\})=0$, it
thus follows that \eqref{eq-deff-qmin} holds for all 
$\phi \in \Np(X) $.  
However, $u \notin \Ndistloc(X)$ as $B(0,1) \cap X \in \Gdist(X)$, 
and thus $u$ is not \p-harmonic on $X$
in the sense of Definition~\ref{def-qmin}.

Had we only required that $u \in \Nploc(X)$, this would instead
have been an example of a positive \p-harmonic function violating
the Harnack inequality $\sup_B u\le C\inf_B u$ 
and the weak maximum principle \eqref{eq-weak-max} below.
On the other hand, by \cite[Theorem~6.2]{BBnoncomp}
the strong maximum principle
would still hold. 
See \cite[Section~6]{BBnoncomp} for further discussion.
\end{example}

If $X$ is proper,
then $\Ndistloc(\Om)=\Nploc(\Om)$ and
\eqref{eq-deff-qmin}
can equivalently be based on 
test functions from $N^{1,p}(X)$ (or $\Lip(X)$) with compact support in $\Om$.
Our definition of quasiminimizers then coincides with the usual
definitions used in the literature, 
see Bj\"orn~\cite[Proposition~3.2]{ABkellogg}.

In the setting of Riemannian manifolds, 
the notion of global \p-harmonic functions 
usually requires the test functions
to be of compact support rather than just vanishing
outside some set $G\in\Gdist(X)$, see e.g.\ 
Cheeger--Gromoll~\cite{ChGro} or Ferrand~\cite[Theorem~4.1]{Fer}. 
Since studies on  Riemannian manifolds often do not focus on understanding
the boundary of the manifold itself, such a class of test functions
is appropriate there.
We are 
interested in the influence of the global structure, 
including the boundary, 
and therefore
use the test function class $\Ndisto(X)$. 
For proper metric spaces and complete connected Riemannian manifolds
these two test function classes coincide, as mentioned above.

The following two examples further illustrate what
can happen when one uses different classes of test functions.

\begin{example}  \label{ex-(-1,1)}
Let $X=\R^{n-1} \times (0,\infty)$
be equipped with the Euclidean distance and the Lebesgue measure.
Note that this is a locally compact space.
Let 
\[
\Om=(-1,1)^{n-1}\times (0,1) \quad \text{and} \quad  u(x)=x_n.
\]
Then $u$ is not \p-harmonic in $\Om$ (seen
as a subset of $X$) because the restriction $v$ to
$\Om$ of the unique \p-harmonic function in $(-1,1)^n\subset\R^n$ with boundary
data $f(x):=|x_n|$ will have smaller energy on $\Om$ than $u$,
since $f$ is not \p-harmonic in $(-1,1)^n$.
Indeed, the function $\phi=v-u$ belongs to $\Ndisto(\Om)$
and can thus be used in Definition~\ref{def-qmin}, even though
it does not have compact support in $\Om$.

In fact, one can see that \p-harmonic functions in subsets of $X$,
as defined in Definition~\ref{def-qmin},
satisfy a zero Neumann boundary condition on the ``missing boundary''
$\R^{n-1} \times \{0\}$,
while this is not in general true for \p-harmonic functions 
defined using 
test functions with compact support.
\end{example}

In spaces which are not locally compact the following example
shows that the situation can
get even worse if one uses test 
functions with compact support.

\begin{example} \label{ex-Rn-Qn}
Let $X=\R^n\setm\Q^n$, $n \ge 2$, equipped with the Lebesgue measure,
and let $p>n$.
Since $\Cone(\Q^n)=0$, we conclude that
$1$-almost no curve in $\R^n$ hits $\Q^n$,
and hence $X$ inherits the global $1$-Poincar\'e inequality from $\R^n$.
It thus follows from
Shanmugalingam~\cite[Theorem~5.1]{Sh-rev} that every $u \in \Np(X)$
has a continuous representative $v \sim u$.
If $u$ has compact support in $X$, then so does $v$, but then $v$
has to be identically $0$ by the density of $\Q^n$.

Hence, using only test functions with compact support
would mean that every continuous Newtonian function is \p-harmonic,
which would violate all types of weak Harnack inequalities as well
as both the weak and strong maximum principles.
\end{example}

The weak maximum principle will be an important
tool in proving Theorem~\ref{thm-intro}\,\ref{a2} in 
Section~\ref{sect-pf-main-a2}.
We will only need it under global assumptions (of doubling and a \p-Poincar\'e
inequality), but
we take the opportunity to deduce it under only local assumptions.
Due to the possible noncompleteness, it does not 
seem to be covered 
in the literature even 
under global assumptions,
nor does it follow from the strong maximum principle, despite
its name, cf.\ Example~\ref{ex-Nploc}.

\begin{thm}  \label{thm-weak-max}
\textup{(Weak maximum principle)}
Assume that $X$ is connected.
If $u$ is quasiharmonic in $\Om$ and $G \in \Gdist(\Om)$, 
with $\emptyset \ne G \ne X$,
then
\begin{equation} \label{eq-weak-max}
    \sup_G u = \sup_{\bdy G} u.
\end{equation}
\end{thm}

Connectedness does not follow
from the local \p-Poincar\'e inequality (in contrast
to the semilocal \p-Poincar\'e inequality) 
and needs to be imposed explicitly.
That connectedness cannot be dropped,
even if we require $\bdy G \ne \emptyset$, follows
by letting 
\[
  X=\Om= \itoverline{B(x,2)} \cup \itoverline{B(y,2)} \subset \R^n,
  \quad \text{where } |x-y| >4,
\]
and $G = X \setm \itoverline{B(y,1)}$,
together with the \p-harmonic function
 $u=\chi_{\itoverline{B(x,2)}}$.

\begin{proof}
As $X$ is connected and $\emptyset \ne G \ne X$, the boundary $\bdy G \ne \emptyset$.
By continuity, $ \sup_G u \ge  \sup_{\bdy G} u \ne -\infty$,
and there is nothing to prove if $\sup_{\bdy G} u= \infty$.
By adding a constant, we may thus assume that $\sup_{\bdy G} u=0$.
Let 
\[
A=\{x \in G: u(x) > 0\}
\quad \text{and} \quad
   \phi=\begin{cases}
       \up, & \text{in } \itoverline{G}, \\
       0, & \text{in } X \setm G. 
     \end{cases}
\]
Then $\phi$ is continuous in $X$ since $G \in \Gdist(\Om)$, and
$\phi \in \Np_0(A) \subset \Ndisto(\Om)$.
Hence
\[
     \int_{A} g^p_{\phi} \,d\mu 
     = \int_{A} g^p_{u} \,d\mu 
     \le Q  \int_{A} g^p_{u-\phi} \,d\mu
     =0,
\]
and so $g_\phi=0$ a.e.\ in $A$.
Since $\phi=0$ outside $A$, we also see that $g_\phi=0$ a.e.\ in $X \setm A$,
and thus a.e.\ in $X$.
It follows from the local \p-Poincar\'e inequality
and the continuity of $\phi$
that $\phi$ is locally constant.
As $X$ is connected,
$\phi$
is constant in $X$.
In particular, $\up=\phi$ is constant in 
$\itoverline{G}$.
Since $\sup_{\bdy G} u=0$, \eqref{eq-weak-max} follows.
\end{proof}

It is also important to know that the suprema in 
\eqref{eq-weak-max}
cannot equal $\infty$.
This follows by continuity of $u$ if $\itoverline{G}$ is compact, and 
by the weak maximum principle (and continuity) if $\bdy G$ is compact.
In general we have the following result.

\begin{prop} \label{prop-bdd-qharm}
Assume that $\mu$ is semilocally doubling
and supports a semilocal \p-Poincar\'e inequality.
If $u$ is quasiharmonic in $\Om$ and $G \in \Gdist(\Om)$,
then $u$ is bounded on $G$.
\end{prop}

The proof 
of this fact under global assumptions in 
Bj\"orn--Marola~\cite[Corollary~8.3]{BMarola}
applies verbatim under semilocal assumptions.
We do not know if this result holds under only 
local assumptions, although
Remark~\ref{rmk-semilocal} implies that it
does if $X$ is in addition proper and connected.

In the rest of this paper we are primarily interested in
global
quasiminimizers
(but for some results in Section~\ref{sect-c}), in which case
certain issues disappear compared with the situation for
arbitrary open subsets of $X$.
Note that 
\[
\Ndisto(X)=\Np_0(X)=\Np(X)
\] 
and so for functions on all of $X$, Definition~\ref{def-qmin}
coincides with several of the other definitions considered in
Bj\"orn--Bj\"orn~\cite{BBnoncomp},
but may still differ from the classical notions of \p-harmonic
functions and quasiminimizers, cf.\ Example~\ref{ex-Nploc}.

Under global assumptions, also in noncomplete spaces, the
following positive Liouville theorem is implied by
the Harnack inequality obtained 
in Kinnunen--Shan\-mu\-ga\-lin\-gam~\cite[Corollary~7.3]{KiSh01}.

\begin{thm} \label{thm-Liouville}
\textup{(Positive Liouville theorem)}
Assume that $\mu$ is globally doubling and supports a global
\p-Poincar\'e inequality.
Then every positive quasiharmonic function on $X$ is constant.
\end{thm}

In bounded spaces the situation is particularly simple,
even under our standing local assumptions.

\begin{prop}  \label{prop-bdd-X}
If $X$ is bounded, then all quasiharmonic functions  on $X$ are 
locally constant, and thus constant in each component.
\end{prop}

\begin{proof}
Let $u$ be quasiharmonic on $X$.
Since $X$ is bounded and $\dist(X,\emptyset)>0$,
we see that 
\[
\Ndistloc(X) = \Np(X) = \Ndisto(X).
\]
Testing~\eqref{eq-deff-qmin} with 
$-u \in \Ndisto(X)$ then yields 
\[
     \int_{u \ne 0} g^p_u \,d\mu \le Q  \int_{u \ne 0} g^p_{u-u} \,d\mu=0.
\]
This, together with the local \p-Poincar\'e inequality 
and the continuity of $u$,
 shows that 
$u$ is locally constant.
\end{proof}

\begin{proof}[Proof of Theorem~\ref{thm-intro}\,\ref{a4}]
Since $X$ supports a global Poincar\'e inequality, it is connected
(see e.g.\ \cite[Proposition~4.2]{BBbook}). Therefore, the theorem
follows directly from Proposition~\ref{prop-bdd-X}. 
\end{proof}

\begin{example} 
That all quasiharmonic functions 
on $X$ are constant can happen even for
unbounded spaces, as seen by letting $X=[0,\infty)$ 
or $X=(0,\infty)$
(equipped with any
locally doubling measure supporting a local \p-Poincar\'e inequality).
To see this, let $0<a< \infty$, $G=\{x \in X : x < a\}$
 and $\phi(x)=u(a)-u(x) \in \Ndisto(G)$.
Then, 
\[
    \int_0^a g_u^p \, d\mu \le Q  \int_0^a g_{u+\phi}^p \, d\mu=0,
\]
and $u$ must be constant in $(0,a)$  for each $a>0$, and thus in $X$.
\end{example}

\section{The proof of Theorem~\ref{thm-intro}\,\ref{a1}}
\label{sect-a}

\emph{In view of Proposition~\ref{prop-bdd-X}, 
we assume in this section that $X$ is unbounded 
and that $\mu$ is globally doubling
      and supports a global \p-Poincar\'e inequality,
with dilation constant $\la$.
We also fix $x_0\in X$  and  set $B_r:=B(x_0,r)$ for 
$r>0$.}

\medskip

The goal of this section is to prove 
Theorem~\ref{thm-intro}\,\ref{a1}.
To do so, we need an energy growth estimate for
quasiharmonic functions in terms of the oscillation of the function.
This estimate will also be crucial when establishing
Theorem~\ref{thm-intro}\,\ref{a2} in the next section.

\begin{lem} \label{lem-key}
Let $u$ be quasiharmonic in a ball $2\la B$.
Then, 
\begin{equation}   \label{eq-key-est}
     \osc_{B} u:=  \sup_{B} u - \inf_{B} u \le \frac{Cr_B}{\mu(B)^{1/p}} 
                \biggl( \int_{2\la B} g_u^p\, d\mu \biggr)^{1/p}.
\end{equation}
In particular,  if the volume growth condition~\eqref{eq-vol-growth}
holds, then there is an
increasing sequence of radii $r_j\to\infty$ such that 
\begin{equation}\label{eq-growth-energy-al} 
 \Bigl(\osc_{B_{r_j}}u\Bigr)^p\le Cr_j^{p-\alpha} 
         \int_{2\lambda B_{r_j}}g_u^p\, d\mu.
\end{equation}
\end{lem}

Lemma~\ref{lem-key-lower} below shows that \eqref{eq-key-est}
is essentially sharp.
If the volume growth condition \eqref{eq-vol-growth}
holds for all sufficiently large radii,
then \eqref{eq-growth-energy-al} holds for 
these radii.

In addition to depending on $X$ and $\mu$, the constant
$C$ above is also allowed to depend on the quasiminimizing 
constant of $u$. 
The same is true for similar estimates in the rest of the paper,
where
$C$ will denote various positive
constants whose values may vary even within a line.

\begin{proof}
Using the weak Harnack inequality 
(see Kinnunen--Shanmugalingam~\cite[Theorem~4.2]{KiSh01} 
or Bj\"orn--Marola~\cite[Theorem~8.2]{BMarola}),
we get that
\[    
\sup_B u \le u_{2B}  + C  \vint_{2B} (u-u_{2B})_\limplus \, d\mu
         \le u_{2B}  + C  \vint_{2B} |u-u_{2B}|\, d\mu.
\]
Applying this to $-u$, we also obtain that 
\[    
 -\inf_B u \le -u_{2B}  + C   \vint_{2B} |u-u_{2B}|\, d\mu.
\]
Combining these two estimates with the Poincar\'e inequality 
gives us 
\[ 
   \osc_{B} u 
       \le C  \vint_{2B} |u-u_{2B}|\, d\mu 
       \le C r_B \biggl(\vint_{2\la B} g_u^p\, d\mu\biggr)^{1/p} 
     \le  \frac{Cr_B}{\mu(B)^{1/p}} 
          \biggl(\int_{2\la B} g_u^p\, d\mu\biggr)^{1/p}.
\] 

The second claim of the lemma now follows directly by
applying~\eqref{eq-vol-growth} to the above inequality.
\end{proof}

\begin{proof}[Proof of Theorem~\ref{thm-intro}\,\ref{a1}]
Given the validity of the volume growth condition~\eqref{eq-vol-growth}, we are able to
apply~\eqref{eq-growth-energy-al}.
Since $u$ has finite energy and $\alp \ge p$, 
letting $j\to\infty$ shows that $u$ is bounded.
Hence $u$ is constant by the positive Liouville 
theorem~\ref{thm-Liouville}. 
\end{proof}

If the volume growth exponent $\alp >p$, then Lemma~\ref{lem-key} 
also provides the following growth rate
for the energy of  
nonconstant quasiharmonic functions. 

\begin{cor}  \label{cor-energy-growth-4}
Let $x_0 \in X$ and 
let $u$ be a nonconstant quasiharmonic function 
on $X$.
If there is $\alp >p$ such 
that 
\begin{equation}    \label{eq-vol-growth-1}
    \limsup_{r \to \infty} \frac{\mu(B_r)}{r^\al} > 0,
\end{equation}
then there is a sequence $r_j \to \infty$ such that 
\[ 
 \int_{B_{r_j}}g_u^p\, d\mu \ge C r_j^{\alp - p}.
\] 
If moreover
\begin{equation}    \label{eq-vol-growth-2}
    \liminf_{r \to \infty} \frac{\mu(B_r)}{r^\al} > 0,
\end{equation}
then 
\begin{equation}   \label{eq-poly-growth-gu-1}
 \int_{B_r}g_u^p\, d\mu \ge C r^{\alp - p}
\quad \text{for all large enough }r.
\end{equation}
\end{cor}

The energy grows also when $\alp=p$, by Theorem~\ref{thm-intro}\,\ref{a1},
but in this case we have no control on how fast it grows.

\begin{proof}
Since $u$ is nonconstant there exists
$R>0$ such that $    \osc_{B_{R/2\la}} u >0$.
By \eqref{eq-vol-growth-1}, there is a sequence $ r_j \nearrow \infty$,
with each $r_j >R$,
such that 
$    \mu(B_{r_j}) \ge Cr_j^\al$.
Thus, by Lemma~\ref{lem-key} and the  doubling property,  
\[
\int_{B_{r_j}} g_u^p\, d\mu
 \ge \frac{C\mu(B_{r_j})}{r_j^p} \Bigl( \osc_{B_{r_j/2\la}} u \Bigr)^p
 \ge Cr_j^{\alp -p}   \Bigl( \osc_{B_{R/2\la}} u \Bigr)^p
 \ge Cr_j^{\alp -p}.
\]

If \eqref{eq-vol-growth-2} holds, then 
$    \mu(B_{r}) \ge Cr^\al$ for all $r >R$
and \eqref{eq-poly-growth-gu-1} follows.
\end{proof}

Note that there may exist nonconstant \p-harmonic functions on $X$ with zero 
oscillation on some ball, see Bj\"orn~\cite[Example~10.1]{ABremove} 
(or \cite[Example~12.24]{BBbook}),
so we need to choose $R$ large enough in the proof above.

There is also a reverse inequality to the one in Lemma~\ref{lem-key}.

\begin{lem} \label{lem-key-lower}
Let $u$ be quasiharmonic in a ball $B$.
Then, 
\[
     \osc_{B} u \ge \frac{C r_B}{\mu(B)^{1/p}}
                \biggl( \int_{\frac{1}{2}B} g_u^p\, d\mu \biggr)^{1/p}.
\]
\end{lem}

\begin{proof}
The Caccioppoli inequality 
(see Kinnunen--Shanmugalingam~\cite[Proposition~3.3]{KiSh01} or 
Bj\"orn--Marola~\cite[Proposition~7.1]{BMarola}) yields
\[ 
         \int_{\frac{1}{2}B} g_u^p \, d\mu 
    \le \frac{C}{r_B^p}  \int_{B} 
        \Bigl(\osc_{B} u \Bigr)^p \, d\mu 
     = \frac{C}{r_B^p} \mu(B) \Bigl( \osc_{B} u \Bigr)^p.
\qedhere
\] 
\end{proof}

\section{The proof of Theorem~\ref{thm-intro}\,\ref{a2}}
\label{sect-pf-main-a2}

\emph{In view of Proposition~\ref{prop-bdd-X}, 
we assume in this section that $X$ is unbounded 
and that $\mu$ is globally doubling
      and supports a global \p-Poincar\'e inequality,
with dilation constant $\la$.
We also fix $x_0\in X$  and  set $B_r:=B(x_0,r)$ for 
$r>0$.}

\medskip

If $\mu$ satisfies~\eqref{eq-vol-growth} with 
$\alpha<p$ and $u$ is a nonconstant quasiharmonic function with
finite energy, then
\eqref{eq-growth-energy-al} 
tells us that the oscillation of $u$ on balls $B_{r_j}$ increases
at most polynomially in $r_j$.
Since 
the volume growth
could be too small in relation to $p$,
the above proof of Theorem~\ref{thm-intro}\,\ref{a1}  does not apply.
In this case we are still able to deduce
the finite-energy Liouville theorem, provided that
a suitable geometric condition is satisfied.
We first define the notion of annular quasiconvexity referred to in the statement of 
Theorem~\ref{thm-intro}\,\ref{a2}.

\begin{deff}\label{def:ann-qcvx}
$X$ is \emph{annularly quasiconvex around $x_0$}
if there exists $\La\ge2$ such that
for every $r>0$, 
each pair of points $x,y  \in B_{2r}\setm B_r$
 can be connected within the annulus $B_{\La r}\setm B_{r/\La}$ 
 by a curve of length at most $\La d(x,y)$.
We say that
$X$ is \emph{annularly quasiconvex}
if it is annularly quasiconvex around
every $x_0\in X$ with $\La$ independent of $x_0$.
\end{deff} 

In certain complete spaces,  annular quasiconvexity follows from a 
global $q$-Poincar\'e inequality for some sufficiently small $q\ge1$,
see Korte~\cite[Theorem~3.3]{Ko07}.
In Lemma~\ref{lem-qconv-Riikka} we show that in similar noncomplete spaces, 
such a global $q$-Poincar\'e inequality implies a discrete analogue of 
annular quasiconvexity, which also implies 
the conclusion of Theorem~\ref{thm-intro}\,\ref{a2}.

\begin{deff}\label{def:seq-ann-LCLS}
$X$ is \emph{sequentially annularly chainable} 
around $x_0$
if there are a constant $\La > 1$
and a 
sequence of radii $r_j\nearrow \infty$
such that 
for every $j$ and $x,y \in \bdy B_{r_j}$, we can find a chain
of points $x=x_1,\ldots, x_m=y$ in
$B_{\La r_j} \setm B_{r_j /\La}$ satisfying 
$d(x_k,x_{k+1}) < r_j/8\lambda \Lambda$ for $k=1,\ldots,m-1$.
\end{deff}

\begin{proof}[Proof of Theorem~\ref{thm-intro}\,\ref{a2}]
This is a direct consequence of 
the following two results.
\end{proof}

\begin{thm} \label{thm-second}
If
$X$ is sequentially annularly chainable around some point $x_0$,
and
$u$ is 
a quasiharmonic function  on $X$ with finite energy, then 
$u$ is constant.
\end{thm}

We postpone the proof 
until after the proof of Lemma~\ref{lem-short-chain-new}.
The following lemma provides us with a sufficient condition for 
sequential annular chainability.

\begin{lem}   \label{lem-qconv-Riikka}
If $X$ supports a global $\sigma$-Poincar\'e inequality
with the dimension exponent
$\sigma>1$ 
as in~\eqref{eq:mass-bound-exp},
or if $X$ is annularly quasiconvex around $x_0$, 
then $X$ is sequentially annularly chainable around $x_0$
\textup{(}for every sequence $r_j \nearrow \infty$\textup{)}.
\end{lem}

\begin{proof}
Let $\Xhat$ be the completion of $X$ taken with respect to the metric $d$
and extend $\mu$ to $\Xhat$ so that $\mu(\Xhat \setm X)=0$.
This zero extension of $\mu$ is a complete Borel regular measure on $\Xhat$, 
by Lemma~3.1 in Bj\"orn--Bj\"orn~\cite{BBnoncomp}.
Proposition~7.1 in Aikawa--Shanmugalingam~\cite{AiSh05} shows that $\mu$
supports a global $\sigma$-Poincar\'e inequality on $\Xhat$.
Moreover, it satisfies~\eqref{eq:mass-bound-exp} with unchanged $s$ and $\sig$.  

Theorem~3.3 in Korte~\cite{Ko07} shows that $\Xhat$ is annularly quasiconvex.
Hence, there exists $\La>1$ such that
every pair $x,y\in \bdy B_r$ can be connected by a curve $\ga$
in $B_{\La r} \setm B_{r/\La}$, 
which provides us with a suitable chain in $\Xhat$.
To conclude the proof, replace each $x_k\in\Xhat$ in the chain by a sufficiently close
point in $X$.

If $X$ is annularly quasiconvex around $x_0$, 
then we can use $X$  instead of $\Xhat$ in the above
discussion to obtain suitable chains in $X$ itself.
\end{proof}

\begin{remark}
A weaker global $q$-Poincar\'e inequality with $q>\sigma$, together
with the global doubling property, implies
that the completion $\Xhat$ is quasiconvex.
Such quasiconvexity is, however, insufficient for our proof. 
Indeed, the space $\{(x_1,x_2)\in\R^2: x_1x_2\ge0\}$,
equipped with the Euclidean metric and the
$2$-dimensional Lebesgue measure, 
supports a global $q$-Poincar\'e inequality for every $q>2$
(see \cite[Example~A.23]{BBbook})
and satisfies~\eqref{eq:mass-bound-exp} with $\sigma=2$,
but is not sequentially annularly chainable. 

Similarly, the examples $X=\R$ and  $X=\R \times [0,1]$ with $\sigma=1$,
demonstrate that 
the sequential annular chainability can fail 
even if $X$ supports a global $1$-Poincar\'e inequality.
Thus the 
global $\sigma$-Poincar\'e inequality in 
Lemma~\ref{lem-qconv-Riikka} cannot
be replaced by a weaker one,
and it is essential that $\sigma >1$.
\end{remark}

The following lemma sets a bound on the effective length of chains
in Definition~\ref{def:seq-ann-LCLS}
and will be used to prove Theorem~\ref{thm-second}.

\begin{lem}   \label{lem-short-chain-new}
Let $\de>0$ and $\La>1$.
Assume that $x=x_1,\ldots, x_m=y$ is a chain 
in $B_{\La r} \setm B_{r /\La}$ satisfying 
$d(x_k,x_{k+1}) < \de r$, $k=1,\dots,m-1$.
Then there is a constant $N_0$, 
depending only on $\de$, $\La$ and the doubling constant, such that
$x$ and $y$ can be connected by a chain of balls 
$\{B^k\}_{k=1}^{N}$, $N \le N_0$, with radii $2\de r$ and centres 
$y_k\in B_{\La r} \setm B_{r /\La}$ so that 
$x \in B^1$, $y \in B^N$
and $B^k \cap B^{k+1}$ is nonempty for $k=1,\ldots,N-1$.

Moreover, $\tau B^k\subset B_{2\La r} \setm B_{r /2\La}$
if $\tau \le1/4\de \La$.
\end{lem}

\begin{proof}
Using the Hausdorff maximality principle and the global doubling condition, 
we can find 
a cover of $B_{\La r} \setm B_{r /\La}$ by at most $N_0$ balls 
\[
\Bh^k=B(\yh_k,\de r) \quad \text{with} \quad
\yh_k \in  B_{\La r} \setm B_{r /\La},
\]
such that $\tfrac12\Bh^k$ are pairwise disjoint,
see for example Heinonen~\cite[Section~10.13]{Hei01}. 
Here $N_0$ depends only on $\de$, $\La$ and the
doubling constant (and in particular is independent of $r$).

For each $l=1,\ldots,m-1$, there exists $k_l$ such that  $x_l \in \Bh^{k_l}$.
It then follows that $x_{l+1} \in 2\Bh^{k_l}$. 
From the sequence $\{2\Bh^{k_l}\}_{l=1}^{m-1}$
we can therefore extract a subsequence $\{B^k\}_{k=1}^{N}$ such that
$x \in B^1$, $y \in B^{N}$, and such that $B^k \cap B^j$ 
is nonempty if and only if $|k-j|\le 1$. 
As it is extracted from the enlargements of balls in the above cover,
we must have $N \le N_0$.

The last inclusion follows directly if $\tau \le1/4\de \La$.
\end{proof}

\begin{proof}[Proof of Theorem~\ref{thm-second}] 
Let $\{r_j\}_{j=1}^\infty$, $x_0$  and $\La$ be as 
in Definition~\ref{def:seq-ann-LCLS}. 
Fix $j$ for which $r_j>8 \la \La$.
We
 can find $x,y \in \bdy B_{r_j}$ so that
\[ 
|u(x)-u(y)| \ge \tfrac{1}{2} \osc_{\bdy B_{r_j}} u.
\]
Let $x=x_1,x_2,\ldots,x_m=y$ be the chain from 
Definition~\ref{def:seq-ann-LCLS}. 
Lemma~\ref{lem-short-chain-new},
with $\de=1/8\la\La$ and $\tau=2\la$, provides us with  
a chain of balls $\{B^k\}_{k=1}^{N}$
of radii $r_j/4\la\La$, such that
\[
   2\la B^k\subset B_{2\La r_j} \setm B_{r_j /2\La}, \quad k=1,\ldots,N,
\]
and $B^k \cap B^{k+1}$ is nonempty for $k=1,\ldots,N-1$, where
$N\le N_0$.
Find $z_k \in B^k \cap B^{k+1}$, $k=1,\ldots,N-1$, 
and let $z_0=x$ and $z_N=y$.
We thus get that, using Lemma~\ref{lem-key},
\begin{align*}
   |u(x) -u(y)| & \le \sum_{k=1}^{N} |u(z_{k-1}) -u(z_k)|
                \le \sum_{k=1}^{N} \osc_{B^k} u  \\
                & \le  \frac{Cr_j}{4\la\La}
                    \sum_{k=1}^{N} \frac{1}{\mu(B^k)^{1/p}} 
                \biggl( \int_{2\la B^k} g_u^p\, d\mu \biggr)^{1/p}.
\end{align*}
Since $\mu$ is globally doubling, we have $\mu(B^k) \simeq \mu(B_{r_j})$ and so
(with $C$ now depending also on $\la$, $\La$ and $N_0$)
\[
   |u(x) -u(y)| \le  \frac{Cr_j}{\mu(B_{r_j})^{1/p}}
    \biggl( \int_{B_{2\La r_j} \setm B_{r_j /2\La}} g_u^p\, d\mu \biggr)^{1/p}.
\]
By Lemma~\ref{lem-key-lower},
\[ 
    \biggl(     \int_{B_{r_j /2}} g_u^p \, d\mu \biggr)^{1/p}
    \le \frac{C}{r_j} \mu(B_{r_j})^{1/p} \osc_{B_{r_j}} u.
\] 
Using the weak maximum principle
(Theorem~\ref{thm-weak-max})
we see that
\[
  \osc_{B_{r_j}} u  =   \osc_{\bdy B_{r_j}} u \le 2  |u(x)-u(y)|.
\]
Combining the last three estimates shows that
\begin{equation} \label{eq-est}
      \biggl( \int_{B_{r_j /2}} g_u^p \, d\mu \biggr)^{1/p}
   \le C  \biggl( \int_{B_{2\La r_j} \setm B_{r_j /2\La}}  g_u^p\, d\mu \biggr)^{1/p}.
\end{equation}

Now, if $u$ has finite energy, then the right-hand side
in \eqref{eq-est} tends to $0$ as $j\to\infty$, 
and hence the left-hand side also
tends to $0$, showing that $g_u=0$ a.e.
The Poincar\'e inequality thus shows that $u$ is constant a.e., and 
since $u$ is continuous it must be constant.
\end{proof}

The estimate~\eqref{eq-est} in the above proof of Theorem~\ref{thm-second}
also provides a growth rate
for the energy of  
nonconstant quasiharmonic functions. 
We express this
for annularly quasiconvex $X$,
in which case the growth is at least polynomial.
If $X$ is only sequentially annularly chainable,
then the growth depends on the corresponding sequence.

\begin{cor}   \label{cor-energy-growth}
If $X$ is annularly quasiconvex around $x_0$
and $u$ is quasiharmonic
on $X$, 
then there is a constant $\beta>0$ such that whenever $0<r<R$,
\begin{equation}   \label{eq-poly-growth-gu}
 \int_{B_r}g_u^p\, d\mu 
\le C\Bigl(\frac{r}{R}\Bigr)^\beta \int_{B_R} g_u^p\, d\mu.
\end{equation}
\end{cor}

If $u$ is nonconstant on $B_{r/\la}$, then $\int_{B_r}g_u^p\, d\mu>0$,
by the \p-Poincar\'e inequality.
Thus from \eqref{eq-poly-growth-gu} we see that
if $X$ is annularly quasiconvex around $x_0$, then 
$\int_{B_R}g_u^p\, d\mu$
must grow at least as fast as $R^\beta$.
Note that there may exist nonconstant \p-harmonic functions on $X$ with zero 
oscillation on some ball, see Bj\"orn~\cite[Example~10.1]{ABremove} 
(or \cite[Example~12.24]{BBbook}).

\begin{proof}
For $r>0$, let
\[
I(r)=\int_{B_r} g_u^p \, d\mu.
\]
Since $X$ is annularly quasiconvex around $x_0$, 
the estimate~\eqref{eq-est} holds for all $r>0$
and hence
\[
I(r/2\La) \le C^p [I(2\La r)-I(r/2\La)].
\]
Adding $C^p I(r/2\La)$ to both sides of the inequality yields that (after replacing
$r/2\La$ by $r$), 
\[
I(r) \le \frac{C^p}{C^p+1} I(4\La^2 r).
\]
Finally, an iteration of this inequality leads to \eqref{eq-poly-growth-gu}
with 
\[
\beta = \frac{\log(1+C^{-p})}{\log 4\La^2}>0.\qedhere
\]
\end{proof}

Corollary~\ref{cor-energy-growth} and the comment
following its statement, together 
with Lemmas~\ref{lem-key} and~\ref{lem-key-lower},
lead to the following estimates which complement the upper
bound~\eqref{eq-growth-energy-al}. 
A similar result
was obtained for harmonic  functions ($p=2$) on certain
weighted Riemannian manifolds, see Wu~\cite[Proposition~2.4]{JWu}.

\begin{cor} \label{cor-osc-beta-growth}
If $X$ is annularly quasiconvex around $x_0$
and $u$ is quasiharmonic on $X$,
then there exists $\beta>0$ such that for all sufficiently large 
$R> r$,
\[
\osc_{B_R} u \ge C \biggl(\frac{R}{r}\biggr)^{1+\beta/p} 
         \biggl(\frac{\mu(B_r)}{\mu(B_R)}\biggr)^{1/p} \osc_{B_r} u
\ge C \biggl(\frac{R}{r}\biggr)^{1+\beta/p-s/p} \osc_{B_r} u,
\]
where $s$ is the dimension exponent from~\eqref{eq:mass-bound-exp}. 
Moreover, if $\mu(B_R)\le CR^p$ for all sufficiently large $R$ and $u$
is nonconstant, then there is $C>0$ such that
\begin{equation}   \label{eq-osc-est-beta}
\osc_{B_R} u \ge C R^{\beta/p}
\quad \text{for sufficiently large } R.
\end{equation}
\end{cor}

If $\mu$ is Ahlfors \p-regular and supports a 
global \p-Poincar\'e inequality, 
$p>1$, then by Korte~\cite[Theorem~3.3]{Ko07},
the assumption of annular quasiconvexity is automatically satisfied, 
and thus \eqref{eq-osc-est-beta} holds in this case.
Also in spaces that are not Ahlfors regular, the estimate 
$\mu(B_R)\le CR^p$ can hold for large $R$. 
For instance, in $\R^n$, equipped with the measure
$d\mu(x)=|x|^{\al}\,dx$ for some
$-n < \alp \le p-n$, the condition $\mu(B_R)\le CR^p$ in
Corollary~\ref{cor-osc-beta-growth} is satisfied for large~$R$,
and so \eqref{eq-osc-est-beta} holds
even though $\mu(B_R)\le CR^p$ fails for small $R$ if 
$n+\alpha<p$, when $x_0=0$.
Note that this measure is globally doubling and supports a global
1-Poincar\'e inequality.

\section{The proofs of Theorems~\ref{thm-intro}\,\ref{a3},
\ref{thm-weighted-R-char} and~\ref{thm-weighted-R-char-intro-2}}
\label{sect-c}

In contrast to $\R^n$, $n \ge 2$, the real line $\R$ is not
annularly quasiconvex, and thus Theorem~\ref{thm-intro}\,\ref{a2}
 is not applicable.
It is well known that the only \p-harmonic functions on unweighted $\R$ are
the linear functions $x \mapsto ax+b$, where $a,b \in \R$ are arbitrary.
From this, both
the positive and finite-energy Liouville theorems
for \p-harmonic functions follow directly.
The positive Liouville theorem for \emph{quasiharmonic} functions 
on unweighted $\R$ is a
special case of Theorem~\ref{thm-Liouville}, but the 
finite-energy Liouville theorem
for quasiharmonic functions requires 
some effort to prove even on unweighted $\R$.

It turns out that this fact can be shown  in greater
generality, namely
on weighted $(\R,\mu)$, where 
$\mu$ is globally doubling and supports
a global \p-Poincar\'e inequality.
Moreover, under only local assumptions, we 
characterize the measures for which the bounded, positive and finite-energy
Liouville theorems hold.
This is the main aim of this section.

We will use 
the following recent characterization
of local assumptions
on $\R$.
Recall that 
$w$ is a \emph{global Muckenhoupt $A_p$ weight} on $\R$, $1<p<\infty$,
if there is a constant $C>0$ such that 
\begin{equation} \label{eq-Ap-cond}
\biggl(\vint_I w\, dx\biggr) 
       \biggl( \vint_{I} w^{1/(1-p)}\,dx \biggr)^{p-1}
< C 
\quad 
\text{for all bounded intervals } I \subset \R.
\end{equation}

\begin{thm} \label{thm-local-padm-R}
\textup{(Bj\"orn--Bj\"orn--Shanmugalingam~\cite[Theorem~1.2
and Proposition~1.3]{BBSpadm})}
The following are equivalent for a measure $\mu$ on $\R$\textup{:}
\begin{enumerate}
\item $\mu$ is locally doubling and supports a local \p-Poincar\'e
inequality on $\R$. 
\item $d\mu=w\,dx$ and for each bounded interval $I\subset \R$ there is 
a global Muckenhoupt $A_p$ weight $\wt$ on $\R$ such that $\wt=w$ on $I$.
\end{enumerate}
Moreover, under the above assumptions, every $u\in \Nploc(\R,\mu)$ 
is locally absolutely continuous on $\R$ and $g_u=|u'|$ a.e.
\end{thm}

As discussed in Section~\ref{sect-qmin}, most
of the general results on \p-harmonic functions and quasiminimizers
are still available under local assumptions, with
the exception of the bounded Liouville theorem.

\medskip

\emph{Throughout 
the rest of
this section, $d\mu = w \, dx$ is a locally 
doubling measure on $\R$ supporting a local \p-Poincar\'e inequality.
In particular, $w>0$ a.e.
We also fix the open subset $\Om=(0,\infty)$ of the metric
measure space $(\R,\mu)$.}

\medskip

On the real line $\R$, 
the dilation constant $\la$ in~\eqref{eq-def-local-PI} can be taken to be $1$,
see \cite[Proposition~3.1]{BBSpadm}.
Note that nonconstant
quasiharmonic functions on $\Om$ and on  $(\R,\mu)$
are strictly monotone, by the strong maximum principle.

\begin{lem} \label{lem-p-harm-R-mu}
A function $u$ is \p-harmonic on 
the open subset $\Om=(0,\infty)$ 
of $(\R,\mu)$
if and only if
there are constants $a,b\in \R$ such that
\begin{equation} \label{eq-a-int}
u(x)=b+a \int_0^x  w^{1/(1-p)} \, dt, \quad x \in \Om.
\end{equation}
Moreover, the energy of $u$ on $\Om$ is
\begin{equation} \label{eq-a-int-2}
   \int_{0}^\infty|u'|^p\,d\mu 
      = |a|^p \int_{0}^\infty w^{1/(1-p)} \, dt,
\end{equation}
which is finite if and only if $u$ is bounded.

The corresponding statements for functions on $(\R,\mu)$
are also true, with the function $u$ given by~\eqref{eq-a-int} 
being \p-harmonic
on $(\R,\mu)$.
\end{lem}

\begin{proof}
Assume first that $u$ is \p-harmonic.
By Theorem~\ref{thm-local-padm-R},
$u$ is locally absolutely continuous on $\Om$ and $g_u=|u'|$ a.e.
We may assume without loss of generality that $u$ is nondecreasing.
Moreover, 
$u$ is a weak solution of the equation 
\[ 
    \Div (w|\nabla u|^{p-2}\nabla u)=0,
\] 
see Heinonen--Kilpel\"ainen--Martio~\cite[Chapter~3]{HeKiMa}.
Thus in this one-dimensional case we see that in the weak sense,
\begin{equation} \label{eq-cc-b}
    (u'(t)^{p-1} w(t))'=0,
\end{equation}
and hence
$u'(t)=aw(t)^{1/(1-p)}$ a.e.\ for some $a \ge 0$
(see H\"ormander~\cite[Theorem~3.1.4]{hormanderI}). 
From this
\eqref{eq-a-int} follows, as $u$ is locally absolutely continuous.

Conversely, if $u$ is given by \eqref{eq-a-int},
then $u$ is locally absolutely continuous on $\Om$ and \eqref{eq-cc-b}
holds, i.e.\ $u$ is \p-harmonic.

Finally, the energy of $u$ is clearly given by
\eqref{eq-a-int-2} and since
the integrands in \eqref{eq-a-int} and \eqref{eq-a-int-2} are the same,
$u$ is bounded if and only if it has finite energy.

The corresponding proof for functions 
on $(\R,\mu)$ is similar.
\end{proof}

In the rest of this section, we fix the function 
\begin{equation} \label{eq-u}
u(x):=\int_0^x w^{1/(1-p)}\, dt, \quad x \in \R,
\end{equation}
which is \p-harmonic by Lemma~\ref{lem-p-harm-R-mu}.
Note that 
\begin{equation}   \label{eq-energy-u-x0-x}
\int_{x_0}^x (u')^p\,d\mu = u(x)-u(x_0).
\end{equation}

\begin{lem}  \label{lem-w-imp-v}
Let 
$v$ be a locally absolutely continuous function 
on $[0,\infty)$ with finite
energy $\int_0^\infty|v'|^p\,d\mu< \infty$.
\begin{enumerate}
\item \label{it-v-bdd}
If $\displaystyle  \int_0^\infty w^{1/(1-p)} \, dt < \infty$,
then $v$  is bounded.
\item \label{it-v-lim-0}
If $\displaystyle  \int_0^\infty w^{1/(1-p)} \, dt = \infty$,
then $v$ satisfies
\begin{equation}   \label{eq-limsup-0}
    \lim_{x \to \infty} \frac{|v(x)|}{u(x)^{1-1/p}} = 0,
\end{equation}
where $u$ is given by \eqref{eq-u}.
\end{enumerate}
\end{lem}

\begin{proof}
By replacing $v$ by $|v|$ if necessary, we may assume that $v \ge 0$.
Statement \ref{it-v-bdd} follows directly by H\"older's inequality, since
\[
|v(x)-v(0)| \le \int_0^x |v'(t)| \, dt 
\le \biggl( \int_0^x |v'|^p \,d\mu \biggr)^{1/p} 
     \biggl( \int_0^x w^{1/(1-p)} \, dt \biggr)^{1-1/p} 
\]
is uniformly bounded for all $x>0$.

To prove \ref{it-v-lim-0}
assume (for a contradiction) that 
$\int_0^\infty w^{1/(1-p)} \, dt = \infty$ and that
there exists $\de>0$ such that 
\[
    \limsup_{x \to \infty} \frac{v(x)}{u(x)^{1-1/p}} > 2 \de.
\]
As $v$ has finite energy, there is $x_0>0$ such that
\begin{equation}   \label{eq-energy-le-mp}
   \int_{x_0}^\infty |v'|^p \,d\mu < \de^p.
\end{equation}
By assumption, $\lim_{x\to\infty} u(x)=\infty$.
Hence there exists $x_1>x_0$
such that $v(x_1)> 2\de u(x_1)^{1-1/p} >2v(x_0)$ .
In particular,
\begin{equation}   \label{eq-choose-x1}
v(x_1)-v(x_0) > \tfrac12 v(x_1) >\de u(x_1)^{1-1/p}\ge \de (u(x_1)-u(x_0))^{1-1/p}.
\end{equation}

Next, we compare the energy of $v$ with that of $u$ on the interval
$[x_0,x_1]$.
It is easily verified that $v$
has the same boundary values on $[x_0,x_1]$
as the function $a u+b$, where 
\[
a = \frac{v(x_1)-v(x_0)}{u(x_1)-u(x_0)}>0  \quad \text{and} \quad
b = v(x_0)- a u(x_0). 
\]
Since $u$ is \p-harmonic (by Lemma~\ref{lem-p-harm-R-mu}), 
it has minimal energy on these intervals 
and hence, using also \eqref{eq-energy-u-x0-x} and 
\eqref{eq-choose-x1}, we obtain
\[ 
 \int_{x_0}^{x_1} |v'|^p \,d\mu 
\ge a^p \int_{x_0}^{x_1} (u')^p \, d\mu  
= \biggl( \frac{v(x_1)-v(x_0)}{u(x_1)-u(x_0)} \biggr)^p (u(x_1)-u(x_0))
> \de^p. 
\] 
As this contradicts \eqref{eq-energy-le-mp},
it follows that \eqref{eq-limsup-0} is true.
\end{proof}

\begin{lem}  \label{lem-m=0-v-const}
Let $u$ be as in \eqref{eq-u}.
If $v\in C([0,\infty))$ is quasiharmonic on $(0,\infty)$, $v(0)=0$ and
\begin{equation} \label{eq-u-v}
    \liminf_{x \to \infty} \frac{|v(x)|}{u(x)^{1-1/p}} = 0,
\end{equation}
then $v\equiv 0$.
\end{lem}

\begin{proof}
By assumption, there is a sequence $x_j\to\infty$, $x_j>0$, such that
\[
\lim_{j\to\infty} \frac{|v(x_j)|}{u(x_j)^{1-1/p}} = 0.
\]
Since $v(t)$ has the same boundary values on $[0,x_j]$ as the function
$a_j u(t)$, where $a_j = v(x_j)/u(x_j)$,
the quasiminimizing property of $v$, 
together with \eqref{eq-energy-u-x0-x}, yields
\[
\int_0^{x_j} |v'|^p \,d\mu \le Q |a_j|^p \int_0^{x_j} (u')^p \, d\mu  
= Q \frac{|v(x_j)|^p}{u(x_j)^{p-1}} \to 0,
\]
where $Q$ is a quasiminimizing constant of $v$.
Hence, $v'=0$ a.e.\ and as $v$ is locally 
absolutely continuous 
(by Theorem~\ref{thm-local-padm-R}),
it must be constant.
\end{proof}

\begin{proof}[Proof of Theorem~\ref{thm-intro}\,\ref{a3}]
This is a direct consequence of 
the positive Liouville theorem~\ref{thm-Liouville} and the following
result.   
\end{proof}

\begin{prop}  \label{prop-finite-energy-bdd}
Let 
$v$ be quasiharmonic on $(\R,\mu)$ or on the open subset $\Om=(0,\infty)$ 
of $(\R,\mu)$.
Then $v$ has finite energy if and only if it is bounded.
\end{prop}

\begin{proof}
First, consider the case when $v$ is quasiharmonic on $\Om$.
By monotonicity, the limit
\[
\lim_{x\to0} v(x) = v(1) -\int_0^1 v'(t)\,dt
\]
exists (finite or infinite).
H\"older's inequality implies that
\[
\biggl| \int_0^1 v'(t) \, dt \biggr|
\le \biggl( \int_0^1 |v'|^p \,d\mu \biggr)^{1/p} 
     \biggl( \int_0^1 w^{1/(1-p)} \, dt \biggr)^{1-1/p}, 
\]
where the last integral is finite by 
Theorem~\ref{thm-local-padm-R} and the local $A_p$ condition~\eqref{eq-Ap-cond}.
This shows that if $v$ is unbounded at 0 then it has infinite energy.

We can therefore assume that $v\in C([0,\infty))$ and $v(0)=0$.
We consider two exhaustive cases:

1.\ If $\int_0^\infty w^{1/(1-p)} \, dt = \infty$, then the ``only if'' part 
follows from Lemmas~\ref{lem-w-imp-v}\,\ref{it-v-lim-0} 
and~\ref{lem-m=0-v-const}.
To see the ``if" part of the claim, note that 
the definition \eqref{eq-u} of $u$ implies that
$\lim_{x\to\infty}u(x)=\infty$ and hence \eqref{eq-u-v} 
holds whenever
$v$ is bounded. 
Lemma~\ref{lem-m=0-v-const} shows that $v\equiv0$ 
and thus of finite energy. 

2.\ If $\int_0^\infty w^{1/(1-p)} \, dt < \infty$, then the ``only if'' part 
is a direct consequence of Lemma~\ref{lem-w-imp-v}\,\ref{it-v-bdd}.
Conversely, 
assume that $v$ is bounded and nonconstant.
Then, by monotonicity, $\lim_{x\to\infty}v(x)$ exists and is finite.
Since $u$ is bounded (by Lemma~\ref{lem-p-harm-R-mu}), we can,
after multiplication by a constant, assume that
\begin{equation}  \label{eq-equal-limits}
0<\lim_{x\to\infty}v(x) = \lim_{x\to\infty}u(x)<\infty.
\end{equation}
Let $x>0$ be arbitrary.
Since $v$ has the same boundary values on $[0,x]$ as the function
$a u$, where $a=v(x)/u(x)$, the quasiminimizing property of $v$ 
(with a quasiminimizing constant $Q$) yields
\[
\int_0^{x} |v'|^p \,d\mu 
\le Q \biggl| \frac{v(x)}{u(x)} \biggr|^p \int_0^{x} (u')^p \, d\mu.
\]
Since $u$ has finite energy (by Lemma~\ref{lem-p-harm-R-mu})
and in 
view of~\eqref{eq-equal-limits}, letting $x\to\infty$
shows that also 
$v$ has finite energy.

Finally, if $v$ is quasiharmonic on $(\R,\mu)$, then applying the above
to both $(0,\infty)$ and $(-\infty,0)$ yields the result.
\end{proof}

We are now ready to obtain the following characterization,
from which Theorems~\ref{thm-weighted-R-char} 
and~\ref{thm-weighted-R-char-intro-2}
will follow rather directly.

\begin{prop}\label{prop-weighted-Om-char}
The following are equivalent for the open subset $\Om=(0,\infty)$
of the metric measure space $(\R,\mu)$\/\textup{:}
\begin{enumerate}
\item \label{R-bdd}
There exists a bounded nonconstant \p-harmonic function on $\Om$.
\item \label{R-finite}
There exists a nonconstant \p-harmonic function with finite energy 
on $\Om$.
\item \label{R-qmin-bdd}
There exists a  bounded nonconstant quasiharmonic function on $\Om$.
\item \label{R-qmin-finite}
There exists a nonconstant quasiharmonic function with finite energy 
on $\Om$.
\item \label{R-integral}
\[ 
    \int_0^\infty w^{1/(1-p)} \, dt < \infty.
\] 
\end{enumerate}
\end{prop}

\begin{proof}
The equivalences \ref{R-bdd} $\eqv$ \ref{R-finite} $\eqv$ \ref{R-integral}
follow from Lemma~\ref{lem-p-harm-R-mu}, while 
Proposition~\ref{prop-finite-energy-bdd} implies that
\ref{R-qmin-bdd}~$\eqv$ \ref{R-qmin-finite}.
The implication \ref{R-bdd}~$\imp$ \ref{R-qmin-bdd} is trivial.

Finally, to prove that $\neg$\ref{R-integral}~$\imp$ $\neg$\ref{R-qmin-bdd},
let $v$ be a bounded quasiharmonic function on $\Om$
with $v(0):=\lim_{t \to 0} v(t)=0$.
The definition \eqref{eq-u} of $u$ implies that
$\lim_{x\to\infty}u(x)=\infty$, and hence \eqref{eq-u-v}
holds.
We can therefore use  Lemma~\ref{lem-m=0-v-const} to conclude
that $v\equiv0$,
i.e.\ \ref{R-qmin-bdd} fails.
\end{proof}

\begin{proof}[Proof of Theorem~\ref{thm-weighted-R-char}]
That $\mu$ is absolutely continuous follows
from Theorem~\ref{thm-local-padm-R}.
Moreover, any nonconstant quasiharmonic function on $(\R,\mu)$
is strictly monotone by the strong maximum principle.
Thus the implications 
\ref{c-a}~$\imp$ \ref{c-c}~$\imp$ \ref{c-e} 
and 
\ref{c-b}~$\imp$ \ref{c-d}~$\imp$ \ref{c-e} 
follow immediately from applying 
Proposition~\ref{prop-weighted-Om-char} to both $(0,\infty)$ and
 $(-\infty,0)$.
Conversely, Lemma~\ref{lem-p-harm-R-mu} shows that 
\ref{c-e}~$\imp$ \ref{c-a}~$\eqv$ \ref{c-b}.
\end{proof}

\begin{proof}[Proof of Theorem~\ref{thm-weighted-R-char-intro-2}]
\ref{S-integral} $\imp$ \ref{S-bdd}
By Lemma~\ref{lem-p-harm-R-mu}, the function $u$, given by \eqref{eq-u},
 is \p-harmonic on $(\R,\mu)$.
On the other hand, by \ref{S-integral} it is 
bounded from above or below (or both), and thus  
either $a+u$ or $a-u$ is a 
positive nonconstant \p-harmonic function on $(\R,\mu)$
if $a \in \R$ is large enough.

\ref{S-bdd} $\imp$ \ref{S-qmin-bdd}
This is trivial.

\ref{S-qmin-bdd} $\imp$ \ref{S-integral}
That $\mu$ is absolutely continuous follows
from Theorem~\ref{thm-local-padm-R}.
Let $v$ be a positive nonconstant quasiharmonic function on $(\R,\mu)$.
Since $v$ is strictly monotone it is either bounded on $(-\infty,0)$
or on $(0,\infty)$ (or both).
In either case, \ref{S-integral} follows from 
Proposition~\ref{prop-weighted-Om-char}
applied to either $(-\infty,0)$ or $(0,\infty)$. 
\end{proof}

The results above raise the questions of whether, in a  metric space
supporting a locally doubling measure and a local Poincar\'e inequality
there can exist a bounded quasiharmonic (or \p-harmonic) function
with infinite energy, and whether there can exist an unbounded
quasiharmonic (or \p-harmonic) function
with finite energy. 
Both questions have affirmative answers. In the latter case this
is shown in Example~\ref{ex:tree} below, and in the former
case in the following example.

\begin{example}  \label{ex-bdd-inf-energy}
Let $\mu_1$ and $\mu_2$ be locally doubling measures 
on $\R$ supporting local \p-Poincar\'e inequalities.
Then $\mu= \mu_1 \otimes \mu_2$ 
is locally doubling and supports a local \p-Poincar\'e inequality
on $\R^2$, cf.\ Bj\"orn--Bj\"orn~\cite[Theorem~3]{BBtensor}, which
can be proved also under local assumptions.

By Theorem~\ref{thm-local-padm-R},
there are 
weights $w_1$ and $w_2$ such that $d\mu_j=w_j\,dx$, $j=1,2$.
Assume that
\[
    \int_{-\infty}^\infty w_1^{1/(1-p)} \, dt < \infty,
\]
and let $u_1$ be any bounded nonconstant $Q$-quasiharmonic
function on $(\R,\mu_1)$, which exists by
Theorem~\ref{thm-weighted-R-char},
and which has finite energy by 
Proposition~\ref{prop-finite-energy-bdd}.

Extend $u_1$ to $\R^2$ by letting $u(x,y)=u_1(x)$ for $(x,y) \in \R^2$.
Then $u$ is $Q$-quasiharmonic in $(\R^2,\mu)$, by 
Corollary~8 in \cite{BBtensor}.
(When $Q=1$, i.e.\ the \p-harmonic case, this can be
deduced directly from the \p-harmonic equation.)
Since $u$ is bounded, 
it follows that the bounded Liouville theorem fails in
$(\R^2,\mu)$.

Now $g_u(x,y)=g_{u_1}(x)=|u_1'(x)|$ a.e.,
by Theorem~\ref{thm-local-padm-R},
and thus
\[
    \int_{\R^2} g_u^p \, d\mu = \mu_2(\R)\int_{\R} g_{u_1}^p \, d\mu_1,
\]
where the integral on the right-hand side is finite, since
$u_1$ has finite energy. 
Hence $u$ has finite energy if and only if $\mu_2(\R)<\infty$,
in which case also the finite-energy Liouville 
theorem fails in
$(\R^2,\mu)$.
When $\mu_2(\R)=\infty$ (e.g.\ when $\mu_2$ is the Lebesgue measure),
$u$ is an example of a bounded quasiharmonic function with infinite
energy, which is \p-harmonic if $Q=1$.
We do not know if the finite-energy Liouville 
theorem holds in this case.
\end{example}

\section{Further examples in the absence of annular chainability}
\label{sect-further-ex}

\begin{example} \label{ex-RxI}
$X=\R \times [0,1]$ is an example of a space for which 
Theorem~\ref{thm-intro} is not applicable.
We shall show that if $X$ is equipped with the Lebesgue measure 
$dm=dx\,dy$
then every quasiharmonic function $v$ on $X$
with finite energy must be constant.

The main ideas are as in Section~\ref{sect-c},
but extra
care needs to be taken in the $y$-direction.
Let $v$ be a nonconstant quasiharmonic function on $X$
with finite energy.
Recall that $g_v=|\nabla v|$ a.e., see Remark~\ref{rmk-gu}.
For $x\in\R$ let
\[
   t(x)=\min_{0 \le y \le 1} v(x,y)
   \quad \text{and} \quad
   T(x)=\max_{0 \le y \le 1} v(x,y).
\]
As $v$ has finite energy, it follows from Lemma~\ref{lem-key} that
\begin{equation}  \label{eq-T-t-0}
\lim_{x \to \pm\infty} (T(x)-t(x)) 
\le \lim_{x \to \pm\infty} 
   C \biggl( \int_{(x-2\la,x+2\la)\times[0,1]} |\nabla v|^p\,dm \biggr)^{1/p} = 0.
\end{equation}
By the strong maximum principle, $T$ and $t$ are strictly
monotone functions on $\R$, and because of~\eqref{eq-T-t-0} we
can therefore assume that they are both strictly increasing and that 
$t(0)=0$.
We shall now show that 
\begin{equation}   \label{eq-lim-t/x-0}
\lim_{x\to\infty} \frac{T(x)}{x^{1-1/p}} 
   = \lim_{x\to\infty} \frac{t(x)}{x^{1-1/p}} = 0.
\end{equation}
Since \eqref{eq-T-t-0} implies that
Here we actually use that $t(0)=0$ so that $\lim_{x \to \infty} T(x)>0$
\[ 
\lim_{x\to\infty} \frac{t(x)}{T(x)} 
= 1 - \lim_{x\to\infty} \frac{T(x)-t(x)}{T(x)} = 1,
\] 
it suffices to consider the second limit in~\eqref{eq-lim-t/x-0}.
Fix $\de>0$ arbitrary.
As $v$ has finite energy, there is $x_0>0$ such that 
\begin{equation} \label{eq-int-est}
\int_{(x_0,\infty)\times[0,1]}|\grad v|^p\,dm
    < \de^p.
\end{equation}
Assume that 
there exists  $x_1>x_0$ such that 
$t(x_1) > 2\de x_1^{1-1/p} > 2T(x_0)$.
Then for all $y\in[0,1]$,
\begin{equation}   \label{eq-ge-de}
v(x_1,y)-v(x_0,y) \ge t(x_1)-T(x_0) > \tfrac12 t(x_1) 
> \de x_1^{1-1/p}.
\end{equation}
It is easily verified that $v(\,\cdot\,,y)$ has the same boundary values 
on $[x_0,x_1]$ as the function $a(y) x + b(y)$, where
\[
a(y) = \frac{v(x_1,y)-v(x_0,y)}{x_1-x_0}>0
  \quad \text{and} \quad
b(y) = v(x_0,y)- a(y) x_0. 
\]
Since linear functions on $\R$ minimize energy,  we obtain as in the proof 
of Lemma~\ref{lem-w-imp-v} that for each $y\in[0,1]$,
\[ 
\int_{x_0}^{x_1} |\bdy_x v(x,y)|^p \,dx
\ge a(y)^p (x_1-x_0) 
= \frac{(v(x_1,y)-v(x_0,y))^p}{(x_1-x_0)^{p-1}} 
> \de^p,
\] 
where the last estimate uses \eqref{eq-ge-de}.
Integrating over $y\in [0,1]$, 
gives  
\[
\int_0^1 \int_{x_0}^{x_1} |\bdy_x v(x,y)|^p \,dx \,dy 
> \de^p,
\]
which contradicts \eqref{eq-int-est}.
So $\limsup_{x\to\infty} t(x)/x^{1-1/p}\le 2\de$,
and letting ${\de}\to0$ proves \eqref{eq-lim-t/x-0}.
Finally, for $n=1,2,\ldots$\,, let
\[
\Om_n = \{(x,y)\in X: 0<v(x,y) < T(n)\} \supset (0,n)\times[0,1],
\]
which is bounded since $\lim_{x \to \infty} t(x)=\infty$
by 
\eqref{eq-T-t-0} and
the positive Liouville theorem (Theorem~\ref{thm-Liouville}).
Compare the energy of $v$  on $\Om_n$ with the energy of the piecewise
linear function
$v_n =T(n)\max\{0,\min\{1,x/n\}\}$.
Note that $v=v_n$ on $\bdy\Om_n$, 
\[
\bdy_x v_n(x,y) = \frac{T(n)}{n}\chi_{\{0 <x<n\}} \quad \text{a.e. on $X$}
\] 
and $\bdy_y v_n\equiv0$.
Using the quasiharmonicity of $v$ (with a 
quasiminimizing constant~$Q$), 
we thus obtain from~\eqref{eq-lim-t/x-0} that
\[
\int_{(0,n)\times[0,1]} |\grad v|^p \,dm 
      \le \int_{\Om_n} |\grad v|^p \,dm 
      \le Q \int_{\Om_n} |\grad v_n|^p \,dm
= Q \frac{T(n)^p}{n^{p-1}} \to 0,
\]
as $n \to \infty$.
This implies that $\grad v=0$ a.e.\ in $(0,\infty)\times[0,1]$,
and thus, by continuity, $v$ is constant therein.
By the strong maximum principle, $v$ is constant on $X$.
\end{example}

We saw in Theorem~\ref{thm-weighted-R-char}
that on the real line
one can never have an unbounded quasiharmonic function
with finite energy.
The following example shows that there are spaces
which admit unbounded \p-harmonic functions
with finite energy.

\begin{example}\label{ex:tree}
Let $G=(V,E)$ be the infinite binary rooted tree, with root $v_0 \in V$
having degree $2$ and all other vertices having degree $3$.
The edge between two neighbouring vertices $a$ and $b$ will be denoted 
$[a,b]$.
Each edge is considered to be a line segment of length $1$,
which makes $G$ into a  metric tree.
Each vertex, but for the root, has three neighbours: one parent
and two children; the root has two children but no parent.

Fixing one geodesic ray $\gamma=\{v_j\}_{j=0}^\infty$ 
starting at the root  $v_0$ and with $v_{j+1}$ being a child
of $v_j$, 
we equip $G$ with the measure $\mu$ as follows.
On the edge $[v_{j},v_{j+1}]$ we let $d\mu= 2^{-j} \,dm$,
where $m$ is the usual one-dimensional Lebesgue measure.
On edges $[a,b]\in E$ that do not belong to the ray $\ga$,
we let $d\mu= 2^{-k} \,dm$, where
$v_k$ is the unique vertex on the ray $\ga$ that is closest to $[a,b]$.

Because of the uniform bound on the degree, 
the measure $\mu$ is  locally doubling and 
supports a 
local $1$-Poincar\'e inequality with uniform constants $r_0$, $C$ and
$\la$  independent of $x_0$.

A function $u:G \to \R$ is \p-harmonic in the sense of 
Definition~\ref{def-qmin}
if and only if it is linear on each edge and 
\begin{equation} \label{eq-pharm-graph}
   \sum_{b \sim a} |u(b)-u(a)|^{p-2} (u(b)-u(a)) \mu([a,b])=0
\end{equation}
holds for each vertex $a$, where the sum is over all neighbours $b$ of $a$,
see Andersson~\cite{andersson},
Holopainen--Soardi~\cite{HoSo1},
Shanmugalingam~\cite[Lemma~3.3]{Sh-conv}
and Bj\"orn--Bj\"orn~\cite[Lemma~A.27]{BBbook}.

We now construct two nonconstant \p-harmonic functions on $G$
with finite energy, one bounded and one unbounded.
Both functions need to be linear on each edge, so we only need
to define them on the vertices.
We start with the unbounded one.

Let $u(v_j)=j$, $j=0,1,\ldots$\,. 
This defines $u$ on the fixed ray $\ga$.
Each vertex $v_j\in \ga$ has two children $v_{j+1}$ and, say, $v_{j+1}'$.
We let $u(v_1')=-1$ and $u(v_{j}')=u(v_{j-1})+1=j$ if $j \ge 2$.
Since
\[
   \mu([v_{j-1},v_j]) = 2^{1-j} 
   =2\mu([v_{j},v_{j+1}])
   =2\mu([v_{j},v_{j+1}']),
   \quad \text{if } j \ge 1,
\]
and $\mu([v_0,v_1])=\mu([v_0,v_1'])$,
this makes $u$ satisfy the \p-harmonic condition 
\eqref{eq-pharm-graph} at all vertices $v_j\in\ga$.
To define $u$ on the remaining vertices we prescribe its change along 
each of its edges as follows. 
Any vertex $a\notin\ga$ has one parent $b$ and two children
$c$ and $c'$, and the corresponding three edges have equal masses.
Letting 
\begin{equation}   \label{eq-def-u(c)}
   u(c)-u(a)=u(c')-u(a)= -2^{1/(1-p)}(u(b)-u(a))
\end{equation}
recursively makes $u$ satisfy the \p-harmonic condition 
\eqref{eq-pharm-graph} at all vertices, and thus
$u$ is \p-harmonic on $G$.

We will now see that $u$ has finite energy. 
Let $G_j$ be the subgraph of $G$ consisting of $v'_j$ together with
all its descendants and corresponding edges.
Because of \eqref{eq-def-u(c)}, the gradient $g_u$ on the
edges of $G_j$ at distance $k-1$ from $v'_j$
is $(2^{1/(1-p)})^k$, and $u$'s energy on $G_j$ is thus
\[
\int_{G_{j}} g_u^p\, d\mu 
   = 2^{1-j} \sum_{k=1}^\infty 2^k (2^{1/(1-p)})^{kp} 
   = 2^{1-j} \sum_{k=1}^\infty 2^{k/(1-p)}, \quad j=1,2,\ldots.
\]
Since $g_u$ on $\ga$ and the adjacent edges is constant 1,
while the measure behaves like $2^{-j}$,
the total energy on $G$ is thus 
\[
  \int_Gg_u^p\, d\mu 
    = \sum_{j=0}^\infty \biggl( 2^{-j} 
     + 2^{-j} + \int_{G_{j+1}} g_u^p\, d\mu \biggr)
    = 4 + 2 \sum_{k=1}^\infty 2^{k/(1-p)} < \infty,
\]
i.e.\ $u$ has finite energy.
Clearly $u$ is unbounded along the ray $\ga$, while it is bounded on $G_j$
for each $j$, and thus bounded from below.

The following modification produces
a bounded nonconstant \p-harmonic function $\ut$ on $G$.
Let $\ut(v_0)=0$, $\ut(v_j)=1$, $j \ge 1$, 
$\ut=u$ on $G_1$, 
\[
\ut=2^{1/(p-1)}(u-1)+1 \quad \text{on }G_2
\] 
and $\ut\equiv 1$ on $G_j$, $j\ge 3$.
Then $\ut$ is a bounded nonconstant \p-harmonic function on $G$.
Moreover, 
$\int_G g_{\ut}^p \,d \mu \le 2^{p/(p-1)} \int_G g_{u}^p \,d \mu < \infty$,
i.e.\ also $\ut$ has finite energy.
\end{example}

\end{document}